\documentclass[a4paper]{amsart}

\usepackage[utf8]{inputenc}
\usepackage[british]{babel}
\usepackage{amsfonts,amsmath,amssymb,amsthm,mathtools,esint}
\usepackage{tikz}
\usepackage{enumitem}
\usepackage[hidelinks]{hyperref}
\usepackage{cleveref}
\crefname{subsection}{subsection}{subsections}
\numberwithin{equation}{section}
\usepackage{xcolor}
\definecolor{dark-gray}{gray}{0.3}
\usepackage{caption}
\newcommand{\R}{\mathbb{R}}

\renewcommand{\d}[1]{\,\mathrm{d}#1}
\renewcommand{\div}{\mathrm{div}}
\newcommand{\I}{\mathrm{I}}
\newcommand{\vertiii}[1]
{{\left\vert\kern-0.25ex\left\vert\kern-0.25ex\left\vert #1 
\right\vert\kern-0.25ex\right\vert\kern-0.25ex\right\vert}}

\newcommand{\Fcal}{\mathcal{F}}
\newcommand{\Tcal}{\mathcal{T}}

\newcommand{\Mcal}{\mathcal{M}}

\newcommand{\RT}{\mathrm{RT}}
\newcommand{\pw}{\mathrm{pw}}
\newcommand{\D}{\mathrm{D}}
\newcommand{\s}{\mathrm{s}}
\newcommand{\M}{\mathbb{M}}
\newcommand{\osc}{\mathrm{osc}}

\newtheorem{theorem}{Theorem}[section]
\newtheorem{lemma}[theorem]{Lemma}

\newtheorem{corollary}[theorem]{Corollary}

\theoremstyle{remark}
\newtheorem{remark}[theorem]{Remark}
\newtheorem{example}[theorem]{Example}

\newcounter{cntS}
\newcommand{\newcnstS}{%
	\refstepcounter{cntS}%
	\ensuremath{c_{\thecntS}}}
\newcommand{\cnstS}[1]{\ensuremath{c_{\ref{#1}}}}

\newcounter{cntL}

\begin{document}

\title[Discrete weak convex duality of HHO methods]{Discrete weak duality of hybrid high-order methods for convex minimization problems}

\author[N.~T.~Tran]{Ngoc Tien Tran}

\thanks{This work has been supported by the European Union's Horizon 2020 research and innovation programme (project DAFNE, grant agreement No.~891734, and project RandomMultiScales, grant agreement No.~865751}

\address[N.~T.~Tran]{%
	Institut f\"ur Mathematik,
	Universit\"at Augsburg,
	86159 Augsburg}
\email{ngoc1.tran@uni-a.de}

\date{\today}

\keywords{discrete weak convex duality, hybrid high-order, convex minimization, a~priori, a~posteriori, adaptive mesh-refining}

\subjclass{65N12, 65N30, 65Y20}

\begin{abstract}
	This paper derives a discrete dual problem for a prototypical hybrid high-order method for convex minimization problems. The discrete primal and dual problem satisfy a weak convex duality that leads to a priori error estimates with convergence rates under additional smoothness assumptions.
	This duality holds for general polyhedral meshes and arbitrary polynomial degrees of the discretization.
	A novel postprocessing is proposed and allows for a~posteriori error estimates on regular triangulations into simplices using primal-dual techniques. This motivates an adaptive mesh-refining algorithm, which performs superiorly compared to uniform mesh refinements.
\end{abstract}

\maketitle

\section{Introduction}

\subsection{Overview}
Given an open bounded polyhedral Lipschitz domain $\Omega \subset \mathbb{R}^n$ with the boundary $\partial \Omega$ and unit normal vector $\nu$, we consider the convex minimization problem of the energy
\begin{align}
	E(v) \coloneqq \int_\Omega (\Psi(\D v) + \psi(x,v)) \d{x} \quad\text{among } v \in V_\mathrm{D} \coloneqq W^{1,p}_\mathrm{D}(\Omega;\mathbb{R}^m),\label{def:energy}
\end{align}
$p \in (1,\infty)$.
Here, $\Psi : \mathbb{M} \to \mathbb{R} \cup \{+ \infty\}$ with $\mathbb{M} \coloneqq \mathbb{R}^{m \times n}$ denotes a proper lower semicontinuous convex energy density, $\psi : \Omega \times \mathbb{R}^m \to \mathbb{R} \cup \{+\infty\}$ is measurable such that $\psi(x, \bullet)$ is a proper lower semicontinuous convex function for a.~e.~$x \in \Omega$, and $V_\mathrm{D} = W^{1,p}_\mathrm{D}(\Omega;\mathbb{R}^m)$ is the space of functions in the Sobolev space $W^{1,p}(\Omega;\mathbb{R}^m)$ with homogenous boundary data on a non-empty closed part $\Gamma_\mathrm{D} \subset \partial \Omega$ of $\partial \Omega$.
Let $\Psi^*$ (resp.~$\psi^*$) denote the convex conjugate of $\Psi$ (resp.~$\psi$ a.e.~in $\Omega$) and $p' \coloneqq p/(p-1) \in (1,\infty)$ with $1/p + 1/p' = 1$ is the H\"older conjugate of $p$.
The dual problem of \eqref{def:energy} maximizes the dual energy
\begin{align}\label{def:dual-energy}
	E^*(\tau) \coloneqq -\int_\Omega (\Psi^*(\tau) + \psi^*(x, \div \tau)) \d{x}
\end{align}
among $\tau \in W_\mathrm{N} \coloneqq W^{p'}_\mathrm{N}(\div,\Omega;\M)$, the Sobolev space of functions $\tau \in L^{p'}(\Omega;\M)$ with a distributional divergence $\div \tau$ belonging to $L^{p'}(\Omega;\R^m)$ and the normal trace $\tau \nu$ vanishes along the Neumann boundary $\Gamma_\mathrm{N} \coloneqq \partial \Omega \setminus \Gamma_\mathrm{D}$ in a weak sense.
The weak duality $E^*(\tau) \leq E(v)$ holds for any $v \in V_\mathrm{D}$ and $\tau \in W_\mathrm{N}$.

It is well-established that duality relations on the discrete level can be utilized in the error analysis of nonconforming methods \cite{CarstensenLiu2015,Bartels2021,Bartels2021b,CarstensenTran2021}.
Only a few methods in the literature are known to preserve the weak duality on the discrete level,
\begin{align}\label{ineq:discrete-duality}
	\max E_h^*(W_\mathrm{N}(\Mcal)) \leq \min E_h(V_\mathrm{D}(\Mcal))
\end{align}
for an appropriate discretization $V_\mathrm{D}(\Mcal)$ of $V_\mathrm{D}$ (resp.~$W_\mathrm{N}(\Mcal)$ of $W_\mathrm{N}$) and discrete energy $E_h$ (resp.~discrete dual energy $E_h^*$).
Historically, \eqref{ineq:discrete-duality} was first known for a Crouzeix-Raviart finite element method (FEM) for the primal problem and a discrete Raviart-Thomas FEM for the dual problem \cite{Marini1985,CarstensenLiu2015,Bartels2021b}.
The generalized Marini identity even leads to the strong duality $\max E_h^*(W_\mathrm{N}(\Mcal)) = \min E_h(V_\mathrm{D}(\Mcal))$ under additional assumptions on $\Psi$ and $\psi$.
For higher-order approximations, the unstabilized hybrid high-order (HHO) method from \cite{AbbasErnPignet2018} and the classical Raviart-Thomas FEM satisfy \eqref{ineq:discrete-duality} without a discrete strong duality in general \cite{CarstensenTran2021}.
Finally, \eqref{ineq:discrete-duality} holds for
the class of lowest-order discontinuous Galerkin (dG) and Raviart-Thomas FEM from \cite{Bartels2021}.
The analysis of all three aforementioned methods requires a regular triangulation of $\Omega$ into simplices (i.e., no hanging nodes are allowed) because they heavily utilize the structure of Crouzeix-Raviart or Raviart-Thomas functions.

\subsection{Motivation and description of main results}
The method from \cite{Bartels2021} stands out as the only one with a stabilization.
The later is required for the design of hybridizable methods, e.g., HHO or hybrid discontinuous Galerkin (HDG) methods, on general polyhedral meshes.
The foundation for establishing \eqref{ineq:discrete-duality} is an integration by parts on the discrete level and the resulting integrals on the skeleton of the mesh determine the stabilizations. In \cite{Bartels2021}, the dG methodology leads to stabilizations that penalize the average of the trace on the primal and the normal trace on the dual level.
These terms do not possess any convergence rates for smooth functions -- an undesired effect that requires overpenalization and projections onto the Crouzeix-Raviart and the Raviart-Thomas finite element space in the error analysis.
Therefore, it only applies to the lowest-order case on regular triangulations into simplices.
Hybridizable methods, on the other hand, provide further flexibility due to additional degrees of freedom on the skeleton, but can they overcome these restrictions?

This paper provides a clear affirmative answer.
In fact, we consider a prototypical HHO method \cite{DiPietroErnLemaire2014,DiPietroErn2015,DiPietroDroniou2020}, which is also related to the HDG methodology \cite{CockburnDiPietroErn2016}.
The ansatz space $V(\mathcal{M}) \coloneqq P_{k+1}(\mathcal{M};\mathbb{R}^m) \times P_k(\Sigma;\mathbb{R}^m)$ features a split of degrees of freedom into piecewise polynomials of degree at most $k+1$ on the (polyhedral) mesh $\mathcal{M}$ and piecewise polynomials of degree at most $k$ on the skeleton $\Sigma$, the set of faces of $\mathcal{M}$.
The discrete scheme replaces the gradient in \eqref{def:energy} by a reconstruction operator $\D_h : V(\mathcal{M}) \to P_k(\mathcal{M};\mathbb{M})$. In contrast to \cite{CarstensenTran2021}, this operator has a non-trivial kernel and so, a stabilization $\s_h : V(\Mcal) \to \mathbb{R}_{\geq 0}$ is required.
This paper utilizes the nonlinear version of the Lehrenfeld-Sch\"oberl stabilization, but the analysis can be extended to the classical HHO stabilization as well.
Given $r \in (1,\infty)$ with  H\"older conjugate $r' = r/(r-1)$, the discrete primal problem minimizes
\begin{align}\label{def:discrete-energy}
	E_h(v_h) \coloneqq \int_\Omega (\Psi(\D_h v_h) + \psi_h(x, \Pi_\Mcal^k v_\mathcal{M})) \d{x} + \s_h(v_h)/r
\end{align}
among $v_h = (v_\mathcal{M},v_\Sigma) \in V_\mathrm{D}(\mathcal{M})$,
where $V_\mathrm{D}(\mathcal{M}) \subset V(\Mcal)$ is a discrete analogue to $V_\mathrm{D}$ and $\psi_h : \Omega \times \mathbb{R}^m \to \mathbb{R} \cup \{+\infty\}$ is an approximation of $\psi$.
Our analysis departs from a discrete integration by parts involving three different operators: a gradient reconstruction $\D_h : V(\Mcal) \to P_k(\Mcal;\M)$, a divergence reconstruction $\div_h : W(\mathcal{M}) \to P_k(\mathcal{M};\mathbb{R}^m)$, which maps from the ansatz space $W(\mathcal{M}) \coloneqq P_k(\mathcal{M};\M) \times P_k(\Sigma;\mathbb{R}^m)$ of $W^{p'}(\div,\Omega;\M)$ onto the space of piecewise polynomials of degree $k$, and a potential reconstruction $\mathcal{R}_h^* : W(\Mcal) \to P_k(\Mcal;\M)$ in the spirit of \cite{DiPietroDroniouRapetti2020,DiPietroDroniou2023} on the dual level.
This leads to a discrete dual problem that maximizes
\begin{align}\label{def:discrete-dual-energy}
	E_h^*(\tau_h) \coloneqq -\int_\Omega (\Psi^*(\mathcal{R}^*_h \tau_h) + \psi_h^*(x,\div_h \tau_h)) \d{x} - \gamma_h(\tau_h)/r'
\end{align}
among $\tau_h = (\tau_\mathcal{M},\tau_\Sigma) \in W_\mathrm{N}(\mathcal{M})$
with a suitable stabilization $\gamma_h : W(\mathcal{M}) \to \mathbb{R}_{\geq 0}$ and a discrete analogue $W_\mathrm{N}(\mathcal{M}) \subset W(\Mcal)$ of $W_\mathrm{N}$.
We establish that $E_h$ from \eqref{def:discrete-energy} and $E_h^*$ from \eqref{def:discrete-dual-energy} satisfy \eqref{ineq:discrete-duality}.
This duality leads to a~priori error estimates with convergence rates under suitable smoothness assumptions.
Apart from a~priori results, this paper proposes a novel $W^{p'}(\div,\Omega,\M)$ conforming postprocessing $\sigma_0$ for the approximation of the dual variable on regular triangulations into simplices. 
Its construction is based on the discrete Euler-Lagrange equations, which allow for the concept of equilibrated tractions in the HHO methodology \cite{DiPietroDroniou2017,AbbasErnPignet2018}.
This is a cheap alternative to equilibrium techniques \cite{BraessSchoeberl2008,ErnVohralik2015,ErnVohralik2020} in the linear case and
can be used to obtain a~posteriori error estimates via primal-dual gap similar to \cite{Repin2000,Bartels2015,CarstensenLiu2015,CarstensenTran2021,BartelsKaltenbach2023}.

\subsection{Notation}
Standard notation for Sobolev and Lebesgue spaces applies throughout this paper with the abbreviation $V \coloneqq W^{1,p}(\Omega;\R^m)$, $W \coloneqq W^{p'}(\div,\Omega;\M)$, and $\|\bullet\|_p \coloneqq \|\bullet\|_{L^p(\Omega)}$ for any $p \in (1,\infty)$.
In particular, the gradient of scalar-valued (resp.~vector- or matrix-valued) functions is denoted by $\nabla$ (resp.~$\D$).
For any open bounded domain $\omega$, $c_\mathrm{P}(\omega)$ is the constant in the Poincar\'e inequality $\|v\|_{L^p(\omega)} \leq c_\mathrm{P}(\omega)\mathrm{diam}(\omega) \|\nabla v\|_{L^p(\omega)}$ for any $v \in W^{1,p}(\omega)$ with $\int_\omega v \d{x} = 0$.
If $\omega$ is convex, then $c_\mathrm{P}(\omega) \leq 2(p/2)^{1/p}$ \cite[Theorem 1.2]{ChuaWeeden2006} for $1 < p < \infty$ and $c_\mathrm{P}(\omega) = 1/\pi$ for $p = 2$ \cite{Bebendorf2003}.
Given $A, B \in \M = \R^{m \times n}$, $A:B$ denotes the Euclidean scalar product of $A$ and $B$, which induces the Frobenius norm $|A| \coloneqq (A:A)^{1/2}$ in $\M$.
The subdifferential $\partial \varphi(A)$ of a convex function $\varphi : \M \to \R \cup \{+\infty\}$ at $A \in \mathbb{M}$ is
\begin{align}\label{def:subdifferential}
	\partial \varphi(A) \coloneqq \{G \in \mathbb{M} : 0 \leq \varphi(B) - \varphi(A) - G:(B - A) \text{ for any } B \in \mathbb{M}\}.
\end{align}
The convex conjugate $\varphi^* : \M \to \mathbb{R} \cup \{+\infty\}$ of $\varphi$ is defined by
\begin{align*}
	\varphi^*(G) \coloneqq \sup_{A \in \M} (A : G - \varphi(A)) \quad\text{for any } G \in \M.
\end{align*}
Given $g \in \R^m$, the (convex) indicator function $\chi_{g}(a)$ is $0$ if $a = g$ and $+\infty$ if $a \neq g$.
The notation $A \lesssim B$ abbreviates $A \leq CB$ for a generic constant $C$ independent of the mesh-size and $A \approx B$ abbreviates $A \lesssim B \lesssim A$.

\subsection{Outline}
The remaining parts of this paper are organized as follows.
\Cref{sec:skeletal-method} provides detailed information on the discretization including reconstruction operators and stabilizations utilized in \eqref{def:discrete-energy}--\eqref{def:discrete-dual-energy} and throughout this paper.
The discrete weak duality \eqref{ineq:discrete-duality} is established in \Cref{sec:discrete-convex-duality}.
\Cref{sec:a-priori} is devoted to the application of \eqref{ineq:discrete-duality} in the a~priori error analysis, while a~posteriori error estimates are established in \Cref{sec:a-posteriori}.
Some numerical benchmarks conclude this paper.

\section{The HHO methodology}\label{sec:skeletal-method}
This section recalls the HHO methodology from \cite{DiPietroErnLemaire2014,DiPietroErn2015,DiPietroDroniou2020}.

\subsection{Polytopal Mesh}\label{sec:triangulation}
Let $\mathcal{M}$ be a finite collection of closed polytopes of positive volume with overlap of measure zero that covers $\overline{\Omega} = \cup_{K \in \mathcal{M}} K$.
A face $S$ of the mesh $\mathcal{M}$ is a closed connected subset of a hyperplane $H_S$ with positive $(n-1)$-dimensional surface measure such that either (a) there exist $K_+,K_- \in \mathcal{M}$ with $S \subset H_S \cap K_+ \cap K_-$ (interior face) or (b) there exists $K_+ \in \mathcal{M}$ with $S \subset H_S \cap K_+ \cap \partial \Omega$ (boundary face).
Let $\Sigma$ be a finite collection of faces with overlap of $(n-1)$-dimensional surface measure zero that covers the skeleton $\partial \Mcal \coloneqq \cup_{K \in \mathcal{M}} \partial K = \cup_{S \in \Sigma} S$ with the split $\Sigma = \Sigma(\Omega) \cup \Sigma(\partial \Omega)$ into the set of interior faces $\Sigma(\Omega)$ and the set of boundary faces $\Sigma(\partial \Omega)$.
Let $\Sigma_\mathrm{D} \coloneqq \{S \in \Sigma: S \subset \Gamma_\mathrm{D}\}$ (resp.~$\Sigma_\mathrm{N} \coloneqq \Sigma(\partial\Omega) \setminus \Sigma_\mathrm{D}$) denote the set of Dirichlet (resp.~Neumann) faces. For $K \in \mathcal{M}$, $\Sigma(K)$ is the set of all faces of $K$.
The normal vector $\nu_S$ of an interior face $S \in \Sigma(\Omega)$ is fixed beforehand and set $\nu_S \coloneqq \nu|_S$ for boundary faces $S \in \Sigma(\partial \Omega)$.
For $S \in \Sigma(\Omega)$, $K_+(S) \in \Mcal$ (resp.~$K_-(S) \in \Mcal$) denotes the unique cell with $S \subset \partial K_{+}(S)$ (resp.~$S \subset \partial K_-(S)$) and $\nu_{K_+(S)}|_S = \nu_S$ (resp.~$\nu_{K_-(S)}|_S = -\nu_S$).
For $S \in \Sigma(\partial \Omega)$, $K_+(S) \in \Mcal$ is the unique cell with $S \subset \partial K_+(S)$.
The explicit reference to $S$ is omitted if there is no likelihood of confusion.
The jump $[v]_S$ and the average $\{v\}_S$ of any function $v \in W^{1,1}(\mathrm{int}(T_+ \cup T_-);\R^m)$ along $S \in \Sigma(\Omega)$ are defined by $[v]_S \coloneqq v|_{K_+} - v|_{K_-} \in L^1(S;\R^m)$ and $\{v\}_S \coloneqq (v|_{K_+} + v|_{K_-})/2 \in L^1(S;\R^m)$.
If $S \in \Sigma(\partial \Omega)$, then $[v]_S \coloneqq v|_S = v_{K_+}|_S \eqqcolon \{v\}_S$.

For theoretical purposes, let $\vartheta$ denote the mesh regularity parameter of $\Mcal$ associated with a matching simplicial submesh, cf.~\cite[Definition 1.38]{DiPietroErn2012} for a precise definition. The constants in discrete inequalities such as the trace or inverse inequality depend on this parameter.
The notation $c_\mathrm{P} \coloneqq \max_{K \in \Mcal} c_\mathrm{P}(K)$ denotes the best possible local Poincar\'e constant on $\Mcal$. For general polyhedral meshes, $c_\mathrm{P}$ may depend on $\vartheta$.
The differential operators $\D_\pw$ and $\div_\pw$ denote the piecewise version of $\D$ and $\div$ without explicit reference to the underlying mesh.

\subsection{Finite element spaces}\label{sec:fem-spaces}
Given a subset $M \subset \R^n$ of diameter $h_M$, let $P_k(M)$ denote the space of polynomials of degree at most $k$.
The vector- and matrix-valued version of $P_k(M)$ read $P_k(M;\mathbb{R}^m) \coloneqq P_k(M)^m$ and $P_k(M;\mathbb{M}) \coloneqq P_k(M;\mathbb{R}^n)^m$.
This notation applies to other spaces as well.
For any $v \in L^1(M)$, $\Pi_M^k v \in P_k(M)$ denotes the $L^2$ projection of $v$ onto $P_k(M)$. The space of piecewise polynomials of degree at most $k$ with respect to the mesh $\mathcal{M}$ or the faces $\Sigma$ is denoted by $P_k(\mathcal{M})$ or $P_k(\Sigma)$.
Given $v \in L^1(\Omega)$, we define the $L^2$ projection $\Pi_{\Mcal}^k v$ of $v$ onto $P_k(\Mcal)$ by $(\Pi_{\Mcal}^k v)|_K = \Pi_K^k v|_K$ in any cell $K \in \Mcal$.
We recall the stability of the $L^2$ projection $\|\Pi_\Mcal^k \xi\|_p \lesssim \|\xi\|_p$ for any $\xi \in L^p(\Omega)$
from, e.g., \cite[Lemma 3.2]{DiPietroDroniou2017}. A consequence of this is the quasi-best approximation
\begin{align}\label{ineq:quasi-opt-L2}
	\|(1 - \Pi_{\Mcal}^k) \xi\|_p \leq \|\xi - \varphi_k\|_p + \|\Pi_{\Mcal}^k(\varphi_k - \xi)\|_p \lesssim \|\xi - \varphi_k\|_p
\end{align}
for any $\varphi_{k} \in P_k(\Mcal)$, where $\varphi_{k} = \Pi_\Mcal^k \varphi_k$ is utilized in the penultimate inequality.

For a simplex $T$, $\RT_k(T) \coloneqq P_k(T;\R^n) + xP_k(T)$ denotes the space of Raviart-Thomas finite element functions. The piecewise version of this on a regular triangulation $\Tcal$ into simplices is called $\RT^\pw_k(\Tcal)$, while $\RT_k(\Tcal) \coloneqq \RT_k^\pw(\Tcal) \cap W$ denotes the space of globally $W^{p'}(\div,\Omega)$ Raviart-Thomas functions.
The mesh $\Mcal$ gives rise to the piecewise constant function $h_\Mcal \in P_0(\Mcal)$ with $h_\Mcal|_K = h_K$; $h_{\max} \coloneqq \max_{K \in \Mcal} h_K$ is the maximal mesh-size of $\Mcal$.
For any $f \in L^{p'}(\Omega;\R^m)$, the oscillation of $f$ reads $\osc(f,\Mcal) \coloneqq \|h_\Mcal(1 - \Pi_\Mcal^k) f\|_{p'}$.

\subsection{Discrete ansatz spaces}\label{sec:discrete-spaces}
Given a fixed $k \in \mathbb{N}_0$, let $V(\mathcal{M}) \coloneqq P_{k+1}(\mathcal{M};\R^m) \times P_k(\Sigma;\R^m)$ (resp.~$W(\mathcal{M}) \coloneqq P_k(\mathcal{M};\M) \times P_k(\Sigma;\R^m)$) denote the discrete ansatz space for $V$ (resp.~$W$).
For any $v_h = (v_\Mcal, v_\Sigma)$, we use the notation $v_K \coloneqq v_\Mcal|_K$ and $v_S \coloneqq v_\Sigma|_S$ to abbreviate the restriction of $v_\Mcal$ in a cell $K \in \Mcal$ and $v_\Sigma$ along a face $S \in \Sigma$.
The same notation applies to discrete objects in $W(\Mcal)$.
The set of Dirichlet faces $\Sigma_\mathrm{D}$ (resp.~Neumann faces $\Sigma_\mathrm{N}$) gives rise to the subspace $P_k(\Sigma \setminus \Sigma_\mathrm{D};\R^m)$ (resp.~$P_k(\Sigma \setminus \Sigma_\mathrm{N};\R^m)$) of $P_k(\Sigma;\R^m)$ with the convention $v_\Sigma \in P_k(\Sigma \setminus \Sigma_\mathrm{D};\R^m)$ (resp.~$\tau_\Sigma \in P_k(\Sigma \setminus \Sigma_\mathrm{N};\R^m)$) if $v_S \equiv 0$ (resp.~$\tau_S \equiv 0$) on Dirichlet faces $S \in \Sigma_\mathrm{D}$ (resp.~Neumann faces $S \in \Sigma_\mathrm{N}$) to model homogenous Dirichlet boundary condition (resp.~vanishing normal traces).
This defines the discrete spaces $V_\mathrm{D}(\mathcal{M}) \coloneqq P_{k+1}(\mathcal{M};\mathbb{R}^m) \times P_k(\Sigma \setminus \Sigma_\mathrm{D}; \mathbb{R}^m)$ and $W_\mathrm{N}(\mathcal{M}) \coloneqq P_k(\mathcal{M};\mathbb{M}) \times P_k(\Sigma \setminus \Sigma_\mathrm{N};\mathbb{R}^m)$ as subspaces of $V(\Mcal)$ and $W(\Mcal)$.
The interpolation operators $\I_V: V \to V(\mathcal{M})$ and $\I_W : W^{1,1}(\Omega;\M) \to W(\Mcal)$ map $v \in V$ to $\I_V v \coloneqq (\Pi_\Mcal^{k+1} v, (\Pi_S^k v)_{S \in \Sigma}) \in V(\Mcal)$ and $\tau \in W^{1,1}(\Omega;\M)$ to $\I_W \tau \coloneqq (\Pi_\Mcal^k \tau, (\Pi_S^k \tau \nu_S)_{S\in \Sigma}) \in W(\Mcal)$.

\subsection{Reconstruction operators and stabilization}
The reconstruction operators $\D_h$ and $\div_h$ provide an approximation of the differential operators $\D$ and $\div$.\\[-0.8em]
\paragraph{\emph{Gradient reconstruction.}} The gradient reconstruction $\D_h v_h \in P_k(\Mcal;\M)$ of $v_h = (v_\Mcal,v_\Sigma) \in V(\Mcal)$ is the unique solution to, for all $\tau_k \in P_k(\Mcal;\M)$,
\begin{align}
	\int_\Omega \D_h v_h : \tau_k \d{x} = - \int_\Omega v_\Mcal \cdot \div_\pw \tau_k \d{x} + \sum_{S \in \Sigma} \int_S v_S \cdot [\tau_k \nu_S]_S \d{s}.
	\label{def:gradient-reconstruction}
\end{align}

\paragraph{\emph{Divergence reconstruction.}} The divergence reconstruction $\div_h \tau_h \in P_k(\Mcal;\R^m)$ of $\tau_h = (\tau_\Mcal,\tau_\Sigma) \in W(\Mcal)$ is the unique solution to, for all $v_k \in P_k(\Mcal;\R^m)$,
\begin{align}
	\int_\Omega \div_h \tau_h \cdot \varphi_{k} \d{x} = - \int_\Omega \D_\pw \varphi_{k} : \tau_\Mcal \d{x} + \sum_{S \in \Sigma} \int_S [\varphi_k]_S \cdot \tau_S \d{s}.
	\label{def:divergence-reconstruction}
\end{align}

\paragraph{\emph{Potential reconstruction}}
The potential reconstruction $\mathcal{R}_h v_h \in P_{k+1}(\Mcal;\R^m)$ of $v_h = (v_\Mcal, v_\Sigma) \in V(\Mcal)$ satisfies, for any $\varphi_{k+1} \in P_{k+1}(\Mcal; \R^m)$, that
\begin{align}\label{def:pot-rec-1}
	&\int_\Omega \D_\pw \mathcal{R}_h v_h : \D_\pw \varphi_{k+1} \d{x}\nonumber\\
	&\qquad = - \int_\Omega v_\Mcal \cdot \Delta_\pw \varphi_{k+1} \d{x} + \sum_{S \in \Sigma} \int_S v_S \cdot [\D_\pw \varphi_{k+1} \nu_S]_S \d{s}.
\end{align}
This defines $\mathcal{R}_h v_h$ uniquely up to piecewise constants, which are fixed by
\begin{align}\label{def:pot-rec-2}
	\int_K \mathcal{R}_h v_h \d{x} = \int_K v_K \d{x} \quad\text{for any } K \in \Mcal.
\end{align}
The following lemma states the characteristic commuting property of the reconstruction operators in the HHO methodology \cite{DiPietroErnLemaire2014,DiPietroErn2015,DiPietroDroniou2020}.
\begin{lemma}[commuting property]\label{lem:commuting-property}
	Any $v \in V$ and $\tau \in W^{1,1}(\Omega;\M)$ satisfy $\D_h \I_{V} v = \Pi_\Mcal^k \D v$, $\div_h \I_W \tau = \Pi_\Mcal^k \div \tau$, and the orthogonality $\D_\pw(v - \mathcal{R}_h \I_V v) \perp \D_\pw P_{k+1}(\Mcal;\R^m)$ with respect to the $L^2$ scalar product.
\end{lemma}
The key for establishing the discrete weak duality \eqref{ineq:discrete-duality} is the design of a reconstruction operator on the dual level in the spirit of \cite{DiPietroDroniouRapetti2020,DiPietroDroniou2023}.
The latter can be computed explicitly, although in this paper, it is only used for theoretical purposes.\\[-0.8em]
\paragraph{\emph{Dual potential reconstruction}}
Let $Z(\mathcal{M})$ denote the $L^2$ orthogonal complement of $\D_\pw P_{k+1}(\mathcal{M};\mathbb{R}^m)$ in $P_k(\mathcal{M}; \mathbb{M})$,
\begin{align}\label{def:orthogonal-decomposition}
	P_k(\mathcal{M}; \mathbb{M}) = \D_\pw P_{k+1}(\mathcal{M};\mathbb{R}^m) \oplus Z(\mathcal{M}),
\end{align}
and $\Pi_{Z(\Mcal)}$ is the $L^2$ projection that maps from $P_k(\mathcal{M}; \mathbb{M})$ onto $Z(\Mcal)$.
The dual potential reconstruction $\mathcal{R}_h^* \tau_h \in P_k(\mathcal{M}; \mathbb{M})$ of $\tau_h = (\tau_\Mcal,\tau_\Sigma) \in W(\Mcal)$ is 
the unique function in $P_k(\Mcal;\M)$ that satisfies
\begin{align}\label{def:dual-pot-rec-1}
	\int_\Omega \mathcal{R}_h^* \tau_h : \D_\pw v_{k+1} \d{x} = - \int_\Omega v_{k+1} \cdot \div_h \tau_h \d{x} + \sum_{S \in \Sigma} \int_S [v_{k+1}]_S \cdot \tau_S \d{s}
\end{align}
for any $v_{k+1} \in P_{k+1}(\Mcal;\R^m)$
and the projection property
\begin{align}\label{def:dual-pot-rec-2}
	\Pi_{Z(\Mcal)} \mathcal{R}_h^* \tau_h = \Pi_{Z(\Mcal)} \tau_\Mcal.
\end{align}
\begin{lemma}[well-definedness of $\mathcal{R}^*_h$]
	For any $\tau_h = (\tau_\Mcal,\tau_\Sigma) \in W(\Mcal)$, there exists a unique function $\mathcal{R}_h^* \tau_h \in P_k(\mathcal{M}; \mathbb{M})$ with \eqref{def:dual-pot-rec-1}--\eqref{def:dual-pot-rec-2}.
\end{lemma}
\begin{proof}
	Let $\Phi(\tau_h) \in P_{k+1}(\Mcal;\R^m)/P_0(\Mcal;\R^m)$ be the unique solution to
	\begin{align}\label{def:dual-pot-rec-ibp}
		\int_\Omega \D_\pw \Phi(\tau_h) : \D_\pw v_{k+1} \d{x} = - \int_\Omega v_{k+1} \cdot \div_h \tau_h \d{x} + \sum_{S \in \Sigma} \int_S [v_{k+1}]_S \cdot \tau_S \d{s}
	\end{align}
	for any $v_{k+1} \in P_{k+1}(\Mcal;\R^m)/P_0(\Mcal;\R^m)$. In fact, the right-hand side of \eqref{def:dual-pot-rec-1} vanishes for any piecewise constant test function $v_{k+1} \in P_0(\Mcal;\R^m)$ thanks to the definition \eqref{def:divergence-reconstruction} of $\div_h$. Therefore, it is a (well-defined) linear functional in the quotient space $P_{k+1}(\Mcal;\R^m)/P_0(\Mcal;\R^m)$.
	Thus, $\Phi(\tau_h)$ is the Riesz representation of this linear functional in $P_{k+1}(\Mcal;\R^m)/P_0(\Mcal;\R^m)$ equipped with the piecewise energy scalar product.
	Hence, $\mathcal{R}_h^* \tau_h \coloneqq \D_\pw \Phi(\tau_h) + \Pi_{Z(\Mcal)} \tau_\Mcal$ satisfies \eqref{def:dual-pot-rec-1}--\eqref{def:dual-pot-rec-2}. The uniqueness follows from the uniqueness of $\Phi(\tau_h)$.
\end{proof}


\paragraph{\emph{Stabilization}}
The kernel of $\D_h$ is non-trivial because any $v_h = (v_\Mcal,0) \in V_\mathrm{D}(\Mcal)$ with $v_\Mcal \perp P_{k-1}(\Mcal;\R^m)$ satisfies $\D_h v_h \equiv 0$, but $v_h$ may not be zero. Hence, $\|\D_h \bullet\|_p$ is not a norm in $V_\mathrm{D}(\Mcal)$ and so, a stabilization $\s_h$ is required.
Fix the parameters $r \in (1,\infty)$ and $s \in \mathbb{R}$.
Given any $u_h = (u_\Mcal,u_\Sigma), v_h = (v_\Mcal,v_\Sigma) \in V(\Mcal)$ and $\tau_h = (\tau_\Mcal,\tau_\Sigma) \in W(\Mcal)$, define
\begin{align}
	\s_h(u_h;v_h) &\coloneqq \sum_{K \in \Mcal} \sum_{S \in \Sigma(K)} h_S^{-s} \int_S T_{K,S} u_h \cdot (v_S - \Pi_S^k v_K) \d{s},
	\label{def:stabilization-primal}\\
	\gamma_h(\tau_h) &\coloneqq \sum_{K \in \Mcal} \sum_{S \in \Sigma(K)} h_S^{s/(r-1)} \|\tau_S - (\mathcal{R}_h^* \tau_h)|_K \nu_K\|_{L^{r'}(S)}^{r'}
	\label{def:stabilization-dual}
\end{align}
with the notation, for any $K \in \Mcal$ and $S \in \Sigma(K)$,
\begin{align}\label{def:trace-operator}
	T_{K,S} u_h \coloneqq \Pi_S^k (|u_S - \Pi_S^k u_K|^{r-2} (u_S - \Pi_S^k u_K))
\end{align}
and the abbreviation $\s_h(v_h;v_h) \coloneqq \s_h(v_h)$.
In the quadratic case $r = 2$, $\s_h$ is known under the label \emph{Lehrenfeld-Sch\"oberl stabilization} in the HDG methodology.
While the reconstruction operators and stabilization of this section are defined globally, they can be computed cell-wise and their computation is carried out in parallel.


\section{Discrete weak convex duality}\label{sec:discrete-convex-duality}
This section establishes the discrete weak duality \eqref{ineq:discrete-duality}.
\begin{theorem}[weak duality]\label{thm:weak-duality}
	The discrete energy $E_h$ from \eqref{def:discrete-energy} and the discrete dual energy $E_h^*$ from \eqref{def:discrete-dual-energy} satisfy \eqref{ineq:discrete-duality}.
\end{theorem}
The point of departure is the following result.
\begin{lemma}[discrete integration by parts]\label{lem:discrete-integration-by-parts}
	Any $v_h = (v_\Mcal,v_\Sigma) \in V_\mathrm{D}(\Mcal)$ and $\tau_h = (\tau_\Mcal, \tau_\Sigma) \in W_\mathrm{N}(\Mcal)$ satisfy
	\begin{align}
		&\int_\Omega \D_h v_h : \mathcal{R}_h^* \tau_h \d{x} = \sum_{S \in \Sigma\setminus\Sigma_\mathrm{D}} \int_S (v_S - \Pi_S^k\{v_\Mcal\}_S) \cdot [\mathcal{R}_h^* \tau_h \nu_S]_S \d{s}\nonumber\\
		&\qquad+ \sum_{S \in \Sigma\setminus\Sigma_\mathrm{N}} \int_S \Pi_S^k[v_\Mcal]_S \cdot (\tau_S - \{\mathcal{R}_h^* \tau_h \nu_S\}_S) \d{s} - \int_\Omega v_\Mcal \cdot \div_h \tau_h \d{x}.
		\label{eq:discrete-integration-by-parts}
	\end{align}
\end{lemma}
\begin{proof}
	A piecewise integration by parts shows
	\begin{align}\label{eq:proof-dibp-1}
		&- \int_\Omega v_\Mcal \cdot \div_\pw \mathcal{R}_h^* \tau_h \d{x} = \int_\Omega \D_\pw v_\Mcal : \mathcal{R}_h^* \tau_h \d{x}\nonumber\\
		&\qquad - \sum_{S \in \Sigma\setminus\Sigma_\mathrm{D}} \int_S \{v_\Mcal\}_S \cdot [\mathcal{R}_h^* \tau_h \nu_S]_S \d{s} - \sum_{S \in \Sigma\setminus\Sigma_\mathrm{N}} \int_S [v_\Mcal]_S \cdot \{\mathcal{R}_h^* \tau_h \nu_S\}_S \d{s}.
	\end{align}
	This and the definition \eqref{def:gradient-reconstruction} of the gradient reconstruction $\D_h$ prove
	\begin{align*}
		&\int_\Omega \D_h v_h : \mathcal{R}_h^* \tau_h \d{x} = \int_\Omega \D_\pw v_\Mcal : \mathcal{R}_h^* \tau_h \d{x}\\
		&+ \sum_{S \in \Sigma\setminus\Sigma_\mathrm{D}} \int_S (v_S - \Pi_S^k\{v_\Mcal\}_S) \cdot [\mathcal{R}_h^* \tau_h \nu_S]_S \d{s}
		- \sum_{S \in \Sigma\setminus\Sigma_\mathrm{N}} \int_S [v_\Mcal]_S \cdot \{\mathcal{R}_h^* \tau_h \nu_S\}_S \d{s}.
	\end{align*}
	Since $\tau_S - \{\mathcal{R}_h^* \tau_h \nu_S\}_S \in P_k(S;\R^m)$ along $S \in \Sigma$, the combination of this with the definition \eqref{def:dual-pot-rec-1} of $\mathcal{R}^*_h$ concludes the proof.
\end{proof}
\begin{proof}[Proof of \Cref{thm:weak-duality}]
	Let any $v_h = (v_\Mcal,v_\Sigma) \in V_\mathrm{D}(\Mcal)$ and $\tau_h = (\tau_\Mcal,\tau_\Sigma) \in W_\mathrm{N}(\Mcal)$ be given.
	Since $\Psi$ is convex, the Fenchel-Young inequality \cite[p.~105]{Rockafellar1970} proves $\D_h v_h : \mathcal{R}_h^* \tau_h \leq \Psi(\D_h v_h) + \Psi^*(\mathcal{R}_h^* \tau_h)$ pointwise a.e.~in $\Omega$.
	An integration of this over the domain $\Omega$ provides
	\begin{align}\label{ineq:proof-discrete-duality-integration}
		0\leq \int_\Omega (\Psi(\D_h v_h) + \Psi^*(\mathcal{R}_h^* \tau_h) - \D_h v_h : \mathcal{R}_h^* \tau_h) \d{x}.
	\end{align}
	Elementary algebra shows the identities
	\begin{align*}
		(v_S &- \Pi_S^k\{v_\Mcal\}_S) \cdot [\mathcal{R}_h^* \tau_h \nu_S]_S + \Pi_S^k[v_\Mcal]_S \cdot (\tau_S - \{\mathcal{R}_h^* \tau_h \nu_S\}_S)\\
		& = (v_S - \Pi_S^k v_{K_-}) \cdot (\tau_S - (\mathcal{R}_h^* \tau_h)|_{K_-} \nu_S) - (v_S - \Pi_S^k v_{K_+}) \cdot (\tau_S - (\mathcal{R}_h^* \tau_h)|_{K_+} \nu_S)
	\end{align*}
	on any interior face $S \in \Sigma(\Omega)$,
	\begin{align*}
		\Pi_S^k[v_\Mcal]_S \cdot (\tau_S - \{\tau_\Mcal \nu_S\}_S) = - (v_S - \Pi_S^k v_\Mcal) \cdot (\tau_S - \mathcal{R}_h^* \tau_h \nu_S)
	\end{align*}
	on any Dirichlet face $S \in \Sigma_\mathrm{D}$, and
	\begin{align*}
		(v_S - \Pi_S^k\{v_\Mcal\}_S) \cdot [\mathcal{R}_h^* \tau_h \nu_S]_S = -(v_S - \Pi_S^k v_\Mcal) \cdot (\tau_S - \mathcal{R}_h^* \tau_h \nu_S)
	\end{align*}
	on any Neumann face $S \in \Sigma_\mathrm{N}$.
	Thus, the first two integrals on the right-hand side of \eqref{eq:discrete-integration-by-parts} are equal to
	\begin{align*}
		- \sum_{K \in \mathcal{M}} \sum_{S \in \Sigma(K)} (\nu_S \cdot \nu_K) \int_S (v_S - \Pi_S^k v_K) \cdot (\tau_S - (\mathcal{R}_h^* \tau_h)|_K \nu_S) \d{s}.
	\end{align*}
	This is bounded by $\s_h(v_h)/r + \gamma_h(\tau_h)/r'$ due to
	the H\"older, Cauchy, and Young inequalities.
	Hence, \eqref{eq:discrete-integration-by-parts} and $\div_h \tau_h \in P_k(\Mcal;\R^m)$ provide
	\begin{align*}
		- \int_\Omega \D_h v_h : \mathcal{R}_h^* \tau_h \d{x} \leq \int_\Omega \Pi_\Mcal^k v_\Mcal \cdot \div_h \tau_h \d{x} + \s_h(v_h)/r + \gamma_h(v_h)/r'.
	\end{align*}
	The combination of this with \eqref{ineq:proof-discrete-duality-integration}
	results in
	\begin{align*}
		0\leq \int_\Omega (\Psi(\D_h v_h) + \Psi^*(\mathcal{R}_h^* \tau_h) + \Pi_\Mcal^k v_\Mcal \cdot \div_h \tau_h) \d{x} + \frac{1}{r}\s_h(v_h) + \frac{1}{r'}\gamma_h(\tau_h).
	\end{align*}
	This and $\Pi_\Mcal^k v_\Mcal \cdot \div_h \tau_h \leq \psi_h(x,\Pi_\Mcal^k v_\Mcal) + \psi_h^*(x,\div_h \tau_h)$ for a.e.~$x \in \Omega$ from the Fenchel-Young inequality \cite[p.~105]{Rockafellar1970} conclude the proof.
\end{proof}

The general idea of this paper can be applied to different choices of stabilizations, e.g., the classical HHO stabilization from \cite{DiPietroErnLemaire2014,DiPietroErn2015,DiPietroDroniou2017,DiPietroDroniou2020}.

\begin{example}[classical stabilization]\label{ex:classical-HHO-stab}
	The HHO method from \cite{DiPietroDroniou2017} utilizes the ansatz space $V(\Mcal) = P_k(\Mcal;\R^m) \times P_k(\Sigma;\R^m)$ and a stabilization of the form
	\begin{align}\label{def:classical-HHO-stab}
		\widetilde{\s}_h(v_h) \coloneqq \sum_{K \in \Mcal} \sum_{S \in \Sigma(K)} h_S^{-s} \|\Pi_S^k(v_S - (\mathcal{P}_h v_h)|_K)\|_{L^r(S)}^r
	\end{align}
	for any $v_h = (v_\Mcal, v_\Sigma) \in V(\Mcal)$
	with $\mathcal{P}_h v_h \coloneqq v_\Mcal + (1 - \Pi_\Mcal^k) \mathcal{R}_h v_h \in P_{k+1}(\Mcal;\R^m)$ and the potential reconstruction $\mathcal{R}_h$ from \eqref{def:pot-rec-1}--\eqref{def:pot-rec-2}.
	Observe that $v_\Mcal - \mathcal{P}_h v_h \perp P_k(\Mcal; \R^m)$ and so, it is possible to replace $v_\Mcal$ by $\mathcal{P}_h v_h$ throughout the proof of \Cref{lem:discrete-integration-by-parts}.
	This leads to the discrete integration by parts formula
	\begin{align*}
		&\int_\Omega \D_h v_h : \mathcal{R}_h^* \tau_h \d{x} = \sum_{S \in \Sigma\setminus\Sigma_\mathrm{D}} \int_S (v_S - \Pi_S^k\{\mathcal{P}_h v_h\}_S) \cdot [\mathcal{R}_h^* \tau_h \nu_S]_S \d{s}\nonumber\\
		&\qquad+ \sum_{S \in \Sigma\setminus\Sigma_\mathrm{N}} \int_S \Pi_S^k[\mathcal{P}_h v_h]_S \cdot (\tau_S - \{\mathcal{R}_h^* \tau_h \nu_S\}_S) \d{s} - \int_\Omega v_\Mcal \cdot \div_h \tau_h \d{x}.
	\end{align*}
	From this, we deduce that \eqref{ineq:discrete-duality} holds for $E_h$ from \eqref{def:discrete-energy} with $\s_h$ replaced by $\widetilde{\s}_h$ and $E_h^*$ from \eqref{ineq:discrete-duality} (without any modification).
\end{example}

\section{A priori}\label{sec:a-priori}
This section establishes error estimates using the weak duality \eqref{ineq:discrete-duality}.
For simplicity, we assume the explicit representation of the lower-order term
\begin{align}\label{def:lower-order-terms}
	\psi(x,a) \coloneqq -f(x) \cdot a \quad\text{and}\quad \psi_h(x,a) \coloneqq -f_h(x)\cdot a
\end{align}
with a volume force $f \in L^{p'}(\Omega;\R^m)$ and $f_h \coloneqq \Pi_\Mcal^k f \in P_k(\Mcal;\R^m)$.
Then
\begin{align*}
	E^*(\tau) &= -\int_\Omega \Psi^*(\tau) \d{x} - \chi_{-f} (\div \tau) &&\text{for any } \tau \in W_\mathrm{N},\\
	E_h^*(\tau_h) &= -\int_\Omega \Psi^*(\mathcal{R}^*_h \tau_h) \d{x} - \chi_{-f_h}(\div_h \tau_h) &&\text{for any } \tau_h \in W_\mathrm{N}(\Mcal).
\end{align*}
Let $u \in \arg\min E(V_\mathrm{D})$ (resp.~$\sigma \in \arg\max E^*(W_\mathrm{N})$) be a minimizer (resp.~maximizer) of $E$ (resp.~$E^*$) in $V_\mathrm{D}$ (resp.~$W_\mathrm{N}$).
The error analysis of this paper applies to convex minimization problems without a duality gap
\begin{align}\label{def:strong-duality}
	\max E^*(W_\mathrm{N}) = E^*(\sigma) = E(u) = \min E(V_\mathrm{D}) < \infty.
\end{align}
Recall the subdifferential of a convex function from \eqref{def:subdifferential}.
\begin{theorem}[a~priori]\label{thm:error-estimate}
	Suppose that \eqref{def:lower-order-terms}--\eqref{def:strong-duality} hold. If $\sigma \in W^{1,1}(\Omega; \M)$, then any $\xi \in L^{p'}(\Omega; \M)$ with
	$\xi \in \partial \Psi(\D_h \I_V u)$ and $\varrho \in L^p(\Omega; \M)$ with $\varrho \in \partial \Psi^*(\mathcal{R}_h^* \I_W \sigma)$ a.e.~in $\Omega$ satisfy
	\begin{align}\label{ineq:error-estimate}
		E_h(\I_V u) &- \min E_h(V_\mathrm{D}(\Mcal)) \leq -\int_\Omega (1 - \Pi_\Mcal^k) \xi:(1 - \Pi_\Mcal^k) \D u \d{x}\nonumber\\
		& + \int_\Omega (\D_h \I_V u - \varrho) : (\sigma - \mathcal{R}_h^* \I_W \sigma) \d{x} + \int_\Omega (f - f_h) \cdot (u - \mathcal{R}_h \I_V u) \d{x}\nonumber\\
		& + \s_h(\I_V u)/r + \gamma_h(\I_W \sigma)/r' - \sum_{S \in \Sigma \setminus \Sigma_\mathrm{N}} \int_S [\mathcal{R}_h \I_V u]_S \cdot (1 - \Pi_S^k)\sigma \nu_S \d{x}.
	\end{align}
\end{theorem}
\begin{proof}
	The proof departs from the split
	\begin{align}\label{ineq:proof-error-estimate-split}
		E_h(\I_V u) - E_h(u_h) \leq E_h(\I_V u) - E(u) + E^*(\sigma) - E_h^*(\I_W \sigma).
	\end{align}
	Since $0 \leq \Psi(\D u) - \Psi(\D_h \I_V u) - \xi : (\D u - \D_h \I_V u)$ a.e.~in $\Omega$ from $\xi \in \partial \Psi(\D_h \I_V u)$, the projection property $\D_h \I_V u = \Pi_\Mcal^k \D u$ from \Cref{lem:commuting-property} imply
	\begin{align}\label{ineq:proof-convergence-rates-primal}
		E_h(\I_V u) - E(u) \leq - \int_\Omega (\xi : (1 - \Pi_\Mcal^k) \D u - (f - f_h) \cdot u) \d{x} + \s_h(\I_V u)/r&.
	\end{align}
	The property $\varrho \in \partial \Psi^*(\mathcal{R}_h^* \I_W \sigma)$ provides $0 \leq \Psi^*(\sigma) - \Psi^*(\mathcal{R}^*_h \I_W \sigma) - \varrho : (\sigma - \mathcal{R}^*_h \I_W \sigma)$ a.e.~in $\Omega$. Since $\div_h \I_W \sigma = \Pi_\Mcal^k \div \sigma = -f_h$ from \Cref{lem:commuting-property}, this shows
	\begin{align}\label{ineq:proof-convergence-rates-dual}
		E^*(\sigma) - E^*_h(\I_W \sigma) \leq - \int_\Omega \varrho : (\sigma - \mathcal{R}_h^* \I_W \sigma) \d{x} + \gamma_h(\I_W \sigma)/r'.
	\end{align}
	The $L^2$ orthogonality
	$\D_h \I_V u - \D_\pw \mathcal{R}_h \I_V u \perp \D_\pw P_{k+1}(\Mcal;\M)$ from \Cref{lem:commuting-property}, $\D_h \I_V u \in P_k(\Mcal;\M)$, and $\Pi_{Z(\Mcal)}(\Pi_\Mcal^k \sigma - \mathcal{R}_h^* \I_W \sigma) = 0$ from \eqref{def:dual-pot-rec-2} lead to
	\begin{align}\label{eq:proof-error-estimate-ibp-1}
		\int_\Omega (\sigma - \mathcal{R}_h^* \I_W \sigma) : \D_h \I_V u \d{x} = \int_\Omega (\sigma - \mathcal{R}_h^* \I_W \sigma) : \D_\pw \mathcal{R}_h \I_V u \d{x}.
	\end{align}
	The definition \eqref{def:dual-pot-rec-1} of $\mathcal{R}_h^*$, a piecewise integration by parts, $\div \sigma = - f$, and $\div_h \I_W \sigma = - f_h$ prove
	\begin{align}\label{eq:proof-error-estimate-ibp-2}
		\int_\Omega (\sigma &- \mathcal{R}_h^* \I_W \sigma) : \D_h\I_V u \d{x}\nonumber\\
		&\qquad = \int_\Omega (f - f_h) \cdot \mathcal{R}_h \I_V u \d{x}
		+ \sum_{S \in \Sigma \setminus \Sigma_\mathrm{N}} \int_S [\mathcal{R}_h \I_V u]_S \cdot (\sigma - \Pi_S^k \sigma) \nu_S \d{x}.
	\end{align}
	The combination of this with \eqref{ineq:proof-error-estimate-split}--\eqref{ineq:proof-convergence-rates-dual} concludes the proof.
\end{proof}
\begin{remark}[lowest-order case]
	If $k = 0$, then it is possible to choose 
	$\xi \in P_0(\Mcal;\M)$ (because $\D_h \I_V u \in P_0(\Mcal;\M)$).
	In this case, the $L^2$ orthogonality $u - \mathcal{R}_h \I_V u \perp P_0(\Mcal;\R^m)$ from \eqref{def:pot-rec-2} imply the simplified version
	\begin{align*}
		E_h(\I_V u) &- \min E_h(V_\mathrm{D}(\Mcal)) \leq \int_\Omega (\D_h \I_V u - \varrho) : (\sigma - \mathcal{R}_h^* \I_W \sigma) \d{x}\nonumber\\
		& + \s_h(\I_V u)/r + \gamma_h(\I_W \sigma)/r' - \sum_{S \in \Sigma \setminus \Sigma_\mathrm{N}} \int_S [\mathcal{R}_h \I_V u]_S \cdot (1 - \Pi_S^k)\sigma \nu_S \d{x}.
	\end{align*}
\end{remark}
\Cref{thm:error-estimate} applies to minimization problems \eqref{def:energy} with energy densities $\Psi$ that satisfy a two-sided growth condition, i.e., there exist positive constants $0 < \newcnstS\label{cnst:growth-lhs-1} \leq \newcnstS\label{cnst:growth-rhs-1}$ and nonnegative constants $\newcnstS\label{cnst:growth-lhs-2}$, $\newcnstS\label{cnst:growth-rhs-2} \geq 0$ such that
\begin{align}\label{ineq:2-sided-growth}
	\cnstS{cnst:growth-lhs-1}|A|^p - \cnstS{cnst:growth-lhs-2} \leq \Psi(A) \leq \cnstS{cnst:growth-rhs-1}|A|^p + \cnstS{cnst:growth-rhs-2} \quad\text{for any } A \in \M.
\end{align}
An immediate consequence of \eqref{ineq:2-sided-growth} is the a~priori bound $\|\D u\|_p + \|\sigma\|_{p'} \lesssim 1$ on the continuous level.
(For instance, $\|\D u\|_p \leq \newcnstS\label{cnst:Du}$ with the positive root $\cnstS{cnst:Du}$ of the scalar-valued function $t \mapsto \cnstS{cnst:growth-lhs-1} t^p - c_\mathrm{P} t \|f\|_{p'} - \cnstS{cnst:growth-lhs-2}|\Omega|$ \cite{CPlechac1997}.)
Lemma 4.2 from \cite{CarstensenTran2021} implies that
a selection of $\xi \in L^{p'}(\Omega;\mathbb{M})$ and $\varrho \in L^p(\Omega;\mathbb{M})$ in \Cref{thm:error-estimate} with $\xi \in \partial \Psi(\mathrm{D}_h \I_V u)$ and $\varrho \in \partial \Psi^*(\mathcal{R}_h^*\I_W \sigma)$ a.e.~in $\Omega$ is \emph{always} possible.
The two-sided growth \eqref{ineq:2-sided-growth} provides the bound
\begin{align}\label{ineq:growth-DW}
	\|\xi\|_{p'}^{p'} \lesssim 1 + \|\D_h \I_V u\|_p^p \quad\text{and}\quad \|\varrho\|_p^p \lesssim 1 + \|\mathcal{R}_h \I_W \sigma\|_{p'}^{p'},
\end{align}
cf.~\cite[Lemma 2.1(c)]{CarstensenTran2021} for explicit constants hidden in the notation $\lesssim$.

\begin{corollary}[convergence rates]\label{cor:convergence-rates}
	Suppose that \eqref{def:lower-order-terms}--\eqref{def:strong-duality} hold and $\Psi$ satisfy \eqref{ineq:2-sided-growth}.
	Let $r = p$ and $p-1-(k+1) \leq s \leq (k+2)(p-1)$.
	Under the smoothness assumptions $u \in V_\mathrm{D} \cap W^{k+2,p}(\Mcal;\R^m)$ and $\sigma \in W^{1,1}(\Omega;\M) \cap W^{k+1,p'}(\Mcal;\M)$, it holds $E_h(\I_V u) - E_h(u_h) \lesssim h_{\max}^{k+1}$.
\end{corollary}
\begin{proof}
	The proof of $\|\sigma - \mathcal{R}_h^* \I_W \sigma\|_{p'} + \gamma(\I_W \sigma) \lesssim h_{\max}^{k+1}$ can follow the arguments of \cite[Section 6]{DiPietroDroniou2023}. For the convenience of the reader, an alternative proof is provided below. Define $\sigma_k \coloneqq \Pi_\Mcal^k \sigma - \mathcal{R}^*_h \I_W \sigma \in P_k(\Mcal;\M)$.
	Given $\Phi_k \in P_k(\Mcal; \M)$, let $\D_\pw v_{k+1}$ for some $v_{k+1} \in P_{k+1}(\Mcal;\R^m)$ denote the $L^2$ projection of $\Phi_k$ onto $\D_\pw P_{k+1}(\Mcal;\R^m)$.
	Since $\Pi_{Z(\Mcal)} \sigma_k = 0$ by the definition \eqref{def:dual-pot-rec-2} of $\mathcal{R}_h^*$, it holds
	$\int_\Omega \sigma_k : \Phi_k \d{x} = \int_\Omega \sigma_k : \D_\pw v_{k+1} \d{x}$.
	By replacing $\D_h \I_V u$ by $\Phi_{k}$ in \eqref{eq:proof-error-estimate-ibp-1}--\eqref{eq:proof-error-estimate-ibp-2},
	we deduce from the H\"older and the Poincar\'e inequalities that
	\begin{align*}
		\int_\Omega \sigma_k : \Phi_k \d{x} &= - \int_\Omega (1 - \Pi_\Mcal^k) \div \sigma \cdot v_{k+1} \d{x}
		+ \sum_{S \in \Sigma \setminus \Sigma_\mathrm{N}} \int_S [v_{k+1}]_S \cdot (1 - \Pi_S^k) \sigma \nu_S \d{s}\nonumber\\
		&\lesssim \osc(\div \sigma,\Mcal)\|\D_\pw v_{k+1}\|_{p} + \sum_{S \in \Sigma \setminus \Sigma_\mathrm{N}} \int_S [v_{k+1}]_S \cdot (1 - \Pi_S^k) \sigma \nu_S \d{s}.
	\end{align*}
	For any $S \in \Sigma \setminus \Sigma_\mathrm{N}$ and $K \in \Mcal$ with $S \in \Sigma(K)$, \eqref{ineq:quasi-opt-L2}, a discrete trace, and a Poincar\'e inequality reveal $\|(1 - \Pi_S^k) \sigma\|_{L^{p'}(S)} \lesssim \|\sigma - (\Pi_K^k \sigma)|_S\|_{L^{p'}(S)} \lesssim h_K^{1/p}\|\D(1 - \Pi_K^k) \sigma\|_{L^{p'}(K)}$.
	This and the same argument applied to $(1 - \Pi_S^k)[v_{k+1}]_S$ lead to
	\begin{align*}
		\sum_{S \in \Sigma \setminus \Sigma_\mathrm{N}} \int_S [v_{k+1}]_S \cdot (1 - \Pi_S^k) \sigma \nu_S \d{s} \lesssim \|\D_\pw v_{k+1}\|_p\|h_\Mcal\D_\pw(1 - \Pi_\Mcal^k) \sigma\|_{p'}.
	\end{align*}
	Since $\|\sigma_k\|_{p'} = \sup_{\Phi \in L^p(\Omega;\M)\setminus\{0\}} \int_\Omega \sigma_k : \Pi_\Mcal^k \Phi \d{x}/\|\Phi\|_{p}$, the combination of the two previously displayed formula with $\|\D_\pw v_{k+1}\|_p \lesssim \|\Phi_k\|_p \lesssim \|\Phi\|_p$ for $\Phi_k \coloneqq \Pi_{\Mcal}^k \Phi$ from the stability of $L^2$ projections yields
	\begin{align}\label{ineq:proof-convergence-rates-1}
		\|\sigma_k\|_{p'} \lesssim \osc(\div \sigma,\Mcal) + \|h_\mathcal{\Mcal}\D_\pw(1 - \Pi_\Mcal^k) \sigma\|_{p'} \lesssim h_{\max}^{k+1}.
	\end{align}
	The triangle inequality and the stability of the $L^2$ projection $\Pi_S^k$ show $\|\Pi_S^k (\tau - (\mathcal{R}^*_h \I_W \tau)|_K)\nu_S\|_{L^{p'}(S)} \lesssim \|\tau - \Pi_K^k \tau\|_{L^{p'}(S)} + \|\sigma_k|_K \nu_S\|_{L^{p'}(S)}$ for any $K \in \Mcal$ and $S \in \Sigma(K)$. Therefore, the trace inequality, its discrete version, and the Poincar\'e inequality provide
	\begin{align}\label{ineq:proof-convergence-rates-2}
		\gamma_h(\I_W \sigma) \lesssim h_{\max}^{s/(p-1)+p'-1}(\|\D_\pw(1 - \Pi_\Mcal^k) \sigma\|_{p'}^{p'} + \|h_\Mcal^{-1}\sigma_k\|_{p'}^{p'})&\nonumber\\
		\lesssim h^{s/(p-1)+p'(k+1)-1}_{\max}&.
	\end{align}
	A triangle inequality and 
	\eqref{ineq:proof-convergence-rates-1}
	show
	\begin{align}\label{ineq:convergence-rates-sigma-stab}
		\|\sigma - \mathcal{R}_h^* \I_W \sigma\|_{p'} \leq \|(1 - \Pi_\Mcal^k) \sigma\|_{p'} + \|\sigma_k\|_{p'} \lesssim h_\mathrm{max}^{k+1}.
	\end{align}
	The convergence rates of the remaining terms in the error estimate \eqref{ineq:error-estimate} are well established in the HHO literature \cite{DiPietroErn2012,DiPietroErnLemaire2014,DiPietroErn2015,DiPietroDroniou2017,DiPietroDroniou2020} and briefly recalled below. The H\"older inequality, \eqref{ineq:growth-DW}, and $\|\D_h \I_V u\|_p + \|\mathcal{R}_h^* \I_W \sigma\|_{p'} \lesssim \|\D u\|_p + \|h_\Mcal \D_\pw \sigma\|_{p'}$ from \eqref{ineq:proof-convergence-rates-1} and \eqref{ineq:convergence-rates-sigma-stab} imply
	\begin{align}
		-\int_\Omega (1 - \Pi_\Mcal^k) \xi:(1 - \Pi_\Mcal^k) \D u \d{x} + \int_\Omega (\D_h \I_V u - \varrho) : (\sigma - \mathcal{R}_h^* \I_W \sigma) \d{x}\nonumber&\\
		\lesssim \|(1 - \Pi_\Mcal^k) \D u\|_p + \|\sigma - \mathcal{R}_h^* \I_W \sigma\|_{p'} \lesssim h_{\max}^{k+1}&.
	\end{align}
	The convergence rates
	\begin{align}\label{ineq:convergence-rates-stab}
		\s_h(\I_V u) \lesssim h_{\max}^{-s + p - 1} \|\D_\pw(1 - \Pi_{\Mcal}^{k+1}) u\|_p^p \lesssim h_{\max}^{-s + (k+2)p - 1}
	\end{align}
	follow
	from the stability of the $L^2$ projection, the trace, and the Poincar\'e inequality.
	These arguments, $[u]_S = 0$ along $S \in \Sigma\setminus\Sigma_\mathrm{N}$, and the H\"older inequality prove
	\begin{align}\label{ineq:convergence-rates-trace-integrals}
		&- \sum_{S \in \Sigma \setminus \Sigma_\mathrm{N}} \int_S [\mathcal{R}_h \I_V u]_S \cdot (1 - \Pi_S^k)\sigma \nu_S \d{s}\nonumber\\
		&\qquad\lesssim h_{\max}\|\D_\pw(u - \mathcal{R}_h \I_V u)\|_p\|\D_\pw(1 - \Pi_\Mcal^k)\sigma\|_{p'} \lesssim h_{\max}^{2(k+1)}.
	\end{align}
	The regularity assumption on $\sigma$ ensures $f = - \div \sigma \in W^{k,p'}(\Mcal;\M)$. This and a piecewise application of the Poincar\'e inequality show
	\begin{align}\label{ineq:convergence-rates-oscillation}
		&\int_\Omega (f - f_h) \cdot (u - \mathcal{R}_h \I_V u) \d{x} \lesssim \osc(f,\Mcal)\|\D_\pw(u - \mathcal{R}_h \I_V u)\|_p \lesssim h_{\max}^{2(k+1)}.
	\end{align}	
	Since $\min\{-s + (k+2) p - 1, s/(p-1) + p'(k+1) - 1\} \geq k+1$ for $p-1-(k+1) \leq s \leq (k+2)(p-1)$, the combination of \eqref{ineq:convergence-rates-sigma-stab}--\eqref{ineq:convergence-rates-oscillation} concludes the proof.
\end{proof}

If $E_h$ is coercive, then the error estimate \eqref{ineq:error-estimate} also bounds the error arising from the coercivity of $E_h$.
Suppose that $\Psi \in C^1(\M)$ and $\psi(x,\bullet) \in C^1(\R^m)$ for a.e.~$x \in \Omega$.
Let $\Phi : L^p(\Omega;\M) \times L^p(\Omega;\M) \to \mathbb{R}_{\geq 0}$ and $\varphi: L^p(\Omega;\R^m) \times L^p(\Omega; \R^m) \to \mathbb{R}_{\geq 0}$ be given such that any $\alpha, \beta \in L^p(\Omega;\M)$ and $a, b \in L^p(\Omega;\R^m)$ satisfy
\begin{align}
	\Phi(\alpha,\beta) &\leq \int_\Omega (\Psi(\alpha) - \Psi(\beta) - \D \Psi(\beta) : (\alpha - \beta)) \d{x},\label{ineq:assumption-cc-W}\\
	\varphi(a,b) &\leq \int_\Omega (\psi_h(x,a) - \psi_h(x,b) - \nabla_u \psi_h(x,b) \cdot (a-b)) \d{x}.\label{ineq:assumption-cc-psi}
\end{align}
The discrete Euler-Lagrange equations associated with the minimization of \eqref{def:discrete-energy} read, for any $v_h = (v_\Mcal, v_\Sigma) \in V_\mathrm{D}(\Mcal)$,
\begin{align}\label{eq:dELE}
	\int_\Omega (\sigma_\Mcal : \D_h v_h + \D_u \psi_h(x,\Pi_\Mcal^k u_\Mcal) \cdot \Pi_{\Mcal}^k v_\Mcal) \d{x} + \s_h(u_h; v_h) = 0.
\end{align}
The choice
$\alpha \coloneqq \D_h v_h$ and $\beta \coloneqq \D_h u_h$ in \eqref{ineq:assumption-cc-W} provides
\begin{align*}
	\Phi(\D_h u_h, \D_h v_h) \leq \int_\Omega (\Psi(\D_h v_h) - \Psi(\D_h u_h) - \sigma_\Mcal : \D_h(v_h - u_h)) \d{x}
\end{align*}
with the abbreviation $\sigma_\Mcal \coloneqq \Pi_\Mcal^k \D \Psi(\D_h u_h)$.
This and \eqref{eq:dELE} imply
\begin{align}
	\Phi(\D_h u_h, \D_h v_h) \leq E_h(v_h) - E_h(u_h) - \s_h(v_h)/r + \s_h(u_h)/r + \s_h(u_h; v_h - u_h)\label{ineq:proof-cc-discrete}\\
	- \int_\Omega (\psi_h(x,\Pi_{\Mcal}^k v_\Mcal) - \psi_h(x,\Pi_\Mcal^k u_\Mcal) - \nabla_u \psi_h(x, \Pi_\Mcal^k u_\Mcal) \cdot \Pi_\Mcal^k(v_\Mcal - u_\Mcal)) \d{x}.\nonumber
\end{align}
The convexity of the continuously differentiable function $x \in \R^m \mapsto |x|^r$ leads to
\begin{align}\label{ineq:stabilization-convexity}
	0 \leq \s_h(v_h)/r - \s_h(u_h)/r - \s_h(u_h; v_h - u_h).
\end{align}
At this point, we note that $\s_h$ is strongly convex \cite[Lemma 5.4(e)]{CarstensenTran2022}. This allows for additional error control in the stabilization, which is neglected because it is not required in the analysis of this paper.
The combination of \eqref{ineq:proof-cc-discrete}--\eqref{ineq:stabilization-convexity} with \eqref{ineq:assumption-cc-psi}
results in
\begin{align}\label{ineq:assmption-cc}
	e_h(u_h,v_h) \leq E_h(v_h) - E_h(u_h) \quad\text{for any } v_h \in V_\mathrm{D}(\Mcal)
\end{align}
with $e_h(u_h,v_h) \coloneqq \Phi(\D_h u_h, \D_h v_h) + \varphi(\Pi_\Mcal^k u_\Mcal, \Pi_\Mcal^k v_\Mcal)$.
If \eqref{def:lower-order-terms} holds, then $\varphi \equiv 0$.
In this case, the choice $v_h \coloneqq \I_V u$ in \eqref{ineq:assmption-cc} and \Cref{cor:convergence-rates} lead to the convergence rates $e_h(u_h,\I_V u) \lesssim h_{\max}^{k+1}$ under the assumptions of \Cref{cor:convergence-rates}.
It is well known \cite{GlowinskiMarrocco1975,Tyukhtin1982,Chow1989} that the convergence rates in \Cref{cor:convergence-rates} can be improved under additional coercivity assumptions on $\Psi$.
In the following, the application of the arguments from the aforementioned references will be briefly outlined.

\begin{example}[strongly convex with Lipschitz continuous gradient]\label{ex:p-Laplace}
	Suppose that $\Psi \in C^1(\M)$ satisfy \eqref{ineq:2-sided-growth} with $p \geq 2$ and there exist positive constants $\newcnstS\label{cnst:cc}, \newcnstS\label{cnst:cc-primal}$ with
	\begin{align}
		\cnstS{cnst:cc}|\D \Psi(A) - \D \Psi(B)|^2 &\leq (1 + |A|^{p-2} + |B|^{p-2})\nonumber\\
		&\quad\times (\Psi(A) - \Psi(B) - \D \Psi(B) : (A - B)),\label{ineq:cc-dual}\\
		\cnstS{cnst:cc-primal}|A - B|^p &\leq \Psi(A) - \Psi(B) - \D \Psi(B) : (A - B)\label{ineq:cc-primal}
	\end{align}
	for any $A, B \in \M$.
	Then \eqref{ineq:assumption-cc-W} holds with $\Phi(\alpha, \beta) \coloneqq \frac{\cnstS{cnst:cc-dual}}{2}\|\D \Psi(\alpha) - \D \Psi(\beta)\|^2_{p'}/(1 + \|\alpha\|_p^p + \|\beta\|_p^p)^{(2-p')/p'} + \frac{\cnstS{cnst:cc-primal}}{2}\|\alpha - \beta\|^p_p$ and $\newcnstS\label{cnst:cc-dual} \coloneqq \cnstS{cnst:cc}/3$ \cite{CPlechac1997,CarstensenTran2022}. 
	Recall $\varrho$ and $\xi = \D \Psi(\D_h \I_V u)$ from \Cref{thm:error-estimate}. Let the assumptions of \Cref{cor:convergence-rates} hold.
	The combination of \eqref{ineq:assumption-cc-W} with $\Psi(\alpha) - \Psi(\beta) \leq - \D \Psi(\alpha) : (\beta - \alpha)$ a.e.~in $\Omega$ from the convexity of $\Psi$ implies the monotonicity
	\begin{align}\label{ineq:cc-monotone-DW}
		\Phi(\alpha,\beta) \leq \int_\Omega (\D \Psi(\alpha) - \D \Psi(\beta)) :(\alpha - \beta) \d{x} \quad\text{for any } \alpha, \beta \in L^p(\Omega;\M)
	\end{align}
	with two immediate consequences.
	On the one hand, \eqref{ineq:cc-monotone-DW} with $\alpha \coloneqq \D u$, $\beta \coloneqq \D_h \I_V u$, $\xi = \D \Psi(\D_h \I_V u)$, a H\"older inequality, and $\|\D u\|_p + \|\D_h \I_V u\|_p \lesssim 1$ lead to $\|\sigma - \xi\|_{p'} \lesssim \|\D u - \D_h \I_V u\|_p$.
	This,
	$\|(1 - \Pi_\Mcal^k)\xi\|_{p'} \lesssim \|\Pi_\Mcal^k \sigma - \xi\|_{p'}$ from \eqref{ineq:quasi-opt-L2},
	and a triangle inequality provide
	\begin{align}\label{ineq:extra-convergence-rates-1}
		\|(1 - \Pi_{\Mcal}^k) \xi\|_{p'} &\lesssim \|(1 - \Pi_\Mcal^k) \sigma\|_{p'} + \|\sigma - \xi\|_{p'}\nonumber\\
		&\lesssim \|(1 - \Pi_\Mcal^k) \sigma\|_{p'} + \|\D u - \D_h \I_V u\|_p \lesssim h^{k+1}_{\max}.
	\end{align}
	On the other hand, \eqref{ineq:cc-monotone-DW} with $\alpha \coloneqq \D u$, $\beta \coloneqq \varrho$, a H\"older inequality, and $\D \Psi(\varrho) = \mathcal{R}_h^* \I_W \sigma$ from $\varrho \in \partial \Psi^*(\mathcal{R}^*_h \I_W \sigma)$ \cite[Theorem 23.5 (a*)]{Rockafellar1970} imply $\|\D u - \varrho\|_{p} \lesssim \|\sigma - \mathcal{R}^*_h \I_W \sigma\|_{p'}^{1/(p-1)}$.
	This and a triangle inequality provide
	\begin{align}\label{ineq:extra-convergence-rates-2}
		\|\D_h \I_V u - \varrho\|_p &\leq \|(1 - \Pi_{\Mcal}^k) \D u\|_p + \|\D u - \varrho\|_p\nonumber\\
		&\lesssim \|(1 - \Pi_{\Mcal}^k) \D u\|_p + \|\sigma - \mathcal{R}_h^* \I_W \sigma\|_{p'}^{1/(p-1)} \lesssim h_\mathrm{max}^{(k+1)/(p-1)}.
	\end{align}
	The bounds \eqref{ineq:extra-convergence-rates-1}--\eqref{ineq:extra-convergence-rates-2} lead to the convergence rates
	\begin{align}
		-\int_\Omega (1 - \Pi_\Mcal^k) \xi:(1 - \Pi_\Mcal^k) \D u \d{x} \lesssim h_{\max}^{2(k+1)},\label{ineq:suboptimal-rates-1}\\
		\int_\Omega (\D_h \I_V u - \varrho) : (\sigma - \mathcal{R}_h^* \I_W \sigma) \d{x} \lesssim h_{\max}^{p'(k+1)}.\label{ineq:suboptimal-rates-2}
	\end{align}
	This, \eqref{ineq:proof-convergence-rates-2}, and \eqref{ineq:convergence-rates-stab}--\eqref{ineq:convergence-rates-oscillation}
	conclude $e_h(u_h, \I_V u) \leq E_h(\I_V u) - E_h(u_h) \lesssim h_{\max}^{p'(k+1)}$ for any $p-1 \leq s \leq (k+1)(p - p') + p - 1$.
	In the case $1 < p \leq 2$, suppose that the coercivity \eqref{ineq:assumption-cc-W} holds for $\Phi(\alpha,\beta) \coloneqq \newcnstS\|\alpha - \beta\|^2_p/(\|\alpha\|_p^p + \|\beta\|_p^p)^{(2-p)/2} + \newcnstS\|\D \Psi(\alpha) - \D \Psi(\beta)\|_{p'}^{p'}$.
	Then, by similar arguments, $\|\sigma - \xi\|_{p'} \lesssim \|\D u - \D_h \I_V u\|_p^{1/(p'-1)}$ and $\|\D u - \varrho\|_p \lesssim \|\sigma - \mathcal{R}_h^* \I_W \sigma\|_{p'}$. Thus, the convergence rates $e_h(u_h, \I_V u) \leq E_h(\I_V u) - E_h(u_h) \lesssim h_{\max}^{p(k+1)}$ hold for any $p(k+1)(p-2) + p - 1 \leq s \leq p - 1$.
\end{example}

\begin{remark}[$p$-Laplace]\label{rem:p-Laplace}
	Given $f \in L^{p'}(\Omega;\R^m)$, the $p$-Laplace problem seeks the unique solution $u \in V_\mathrm{D}$ to
	$- \div \sigma = f$ in $\Omega$ with $\sigma \coloneqq |\D u|^{p-2} \D u$.
	Then $u$ minimizes $E$ in $V_\mathrm{D}$ with the energy density $\Psi(A) \coloneqq |A|^p/p$ for any $A \in \M$.
	Since $\Psi$ and $\Psi^*$ are strictly convex, the minimizer $u$ of $E$ in $V_\mathrm{D}$ and the maximizer $\sigma \coloneqq \D \Psi (\D u)$ of $E^*$ in $W_\mathrm{N}$ are unique.
	The coercivity assumptions of \Cref{ex:p-Laplace} are satisfied and so, the convergence rates therein hold. They coincide with \cite{GlowinskiMarrocco1975} for $p \geq 2$ and \cite{Tyukhtin1982,Chow1989} for $1 < p \leq 2$.
	The consequence $\|\D_h (\I_V u - u_h)\|_p \lesssim h^{(k+1)p/2}$ of the latter provides an improvement to existing rates in the HHO literature \cite{CarstensenTran2021,DiPietroDroniou2021}.
	The derivation of quadratic convergence rates for $E_h(\I_V u) - E_h(u_h)$
	appears possible using the coercivity of $\Psi$ with respect to the so-called quasi-norm \cite{BarrettLiu1993,DieningKreuzer2008}, cf.~\cite{DieningKoenerRuzickaToulopoulos2014} for the dG method and \cite{Bartels2021b,kaltenbach2022} for the Crouzeix-Raviart FEM.
	This approach, however, is currently restricted to the lowest-order case due to the lack of higher convergence rates for interpolation in Orlicz-Sobolev spaces \cite{DieningRuzicka2007}.
\end{remark}

\section{A~posteriori}\label{sec:a-posteriori}
The main goal of this section is the design of a $W^{p'}(\div,\Omega;\M)$ conforming postprocessing from a computed discrete minimizer $u_h = (u_\Mcal, u_\Sigma)$ of \eqref{def:discrete-energy} for the derivation of primal-dual error estimates.
Recall the abbreviation $\sigma_\Mcal = \Pi_\Mcal^k \D \Psi(\D_h u_h)$ for $\Psi \in C^1(\M)$.
In contrast to the unstabilized HHO method on regular triangulations into simplices from \cite{CarstensenTran2021}, $\sigma_\Mcal \notin W_\mathrm{N}$ has to be expected.
In fact, the discrete Euler-Lagrange equations \eqref{eq:dELE} imply the following result also known as \emph{equilibrium of traces}  \cite{BottiDiPietroSochala2017,AbbasErnPignet2018}.
\begin{corollary}[normal jump of $\sigma_\Mcal$]\label{cor:normal-jump}
	Suppose that $\Psi \in C^1(\M)$ and $\psi_h(x,\bullet) \in C^1(\R^m)$ for a.e.~$x \in \Omega$.
	For any $S \in \Sigma\setminus\Sigma_\mathrm{D}$, $\sigma_\Mcal = \Pi_\Mcal^k \D \Psi(\D_h u_h) \in P_k(\Mcal;\M)$ satisfies, for $T_{K,S}$ from \eqref{def:trace-operator}, that
	\begin{align}
		[\sigma_\Mcal \nu_S]_S = - \sum_{K \in \Mcal, S \in \Sigma(K)} h_S^{-s} T_{K,S} u_h.
		\label{eq:sigma-T-normal-trace}
	\end{align}
\end{corollary}
\begin{proof}
	The discrete Euler-Lagrange equations \eqref{eq:dELE} and the definition \eqref{def:gradient-reconstruction} of $\D_h$ imply, for any $v_h = (v_\Mcal, v_\Sigma) \in V_\mathrm{D}(\Mcal)$, that
	\begin{align}\label{eq:dELE-v2}
		\int_\Omega \Pi_\Mcal^k v_\Mcal \cdot &(\nabla_u \psi_h(x, \Pi_\Mcal^k u_\Mcal) - \div_\pw \sigma_\Mcal) \d{x}\nonumber\\
		& + \sum_{S \in \Sigma\setminus\Sigma_\mathrm{D}} \int_S v_S \cdot [\sigma_\Mcal \nu_S]_S \d{s} + \s(u_h;v_h) = 0.
	\end{align}
	Fix a face $S \in \Sigma\setminus\Sigma_\mathrm{D}$. The choice $v_\Mcal \equiv 0$ and $v_\Sigma \in P_k(\Sigma\setminus\Sigma_\mathrm{D};\mathbb{R}^m)$ with $v_\Sigma|_E \equiv 0$ for all $E \in \Sigma$, $E \neq S$ in \eqref{eq:dELE-v2} provides
	\begin{align*}
		\int_S v_S\cdot([\sigma_\Mcal\nu_S]_S + \sum_{K \in \Mcal, S \in \Sigma(K)} h_S^{-s} T_{K,S} u_h) \d{s} = 0
	\end{align*}
	for all $v_S \in P_k(S;\R^m)$.
	This orthogonality concludes \eqref{eq:sigma-T-normal-trace}.
\end{proof}

Although $\sigma_\Mcal \notin W_\mathrm{N}$ holds in general, it can be utilized to construct an appropriate postprocessing $\sigma_0 \in \RT_k(\mathcal{T};\mathbb{M}) \cap W_\mathrm{N}$ in a matching simplicial submesh $\mathcal{T}$ of $\mathcal{M}$ as follows. (We refer to \cite[Definition 1.37]{DiPietroErn2012} for a precise definition of a matching submesh.)

\begin{lemma}[Raviart-Thomas reconstruction]\label{lem:discrete-dual-variable}
	Suppose that $\Psi \in C^1(\M)$ and $\psi_h(x,\bullet) \in C^1(\R^m)$ for a.e.~$x \in \Omega$. Then the discrete dual variable $\sigma_h = (\sigma_\Mcal, \sigma_\Sigma) \in W_\mathrm{N}(\Mcal)$ with $\sigma_\Mcal = \Pi_\Mcal^k \D \Psi(\D_h u_h) \in P_k(\Mcal;\M)$ and
	\begin{align*}
		\sigma_S \coloneqq \begin{cases}
			\{\sigma_\Mcal\nu_S\}_S + h_S^{-s} (T_{K_+,S} u_h - T_{K_-,S} u_h)/2 \in P_k(S;\R^m) &\mbox{if } S \in \Sigma(\Omega),\\
			\sigma_\Mcal\nu_S + h_S^{-s} T_{K_+,S} u_h \in P_k(S;\R^m) &\mbox{if } S \in \Sigma_\mathrm{D}
		\end{cases}
	\end{align*}
	satisfies $\div_h \sigma_h = \Pi_\Mcal^k \nabla_u \psi_h(x, \Pi_\Mcal^k u_\Mcal) \in P_k(\Mcal;\R^m)$.
	Given a matching simplicial submesh $\Tcal$ of $\Mcal$ with the set $\Fcal$ of faces, let $\sigma_0 \in \RT_k(\Tcal;\M) \cap W_\mathrm{N}$ be the unique Raviart-Thomas finite element function defined by the weights
	\begin{align}\label{def:sigma-RT}
		\begin{split}
			\Pi_\Tcal^{k-1} \sigma_0 &= \Pi_\Tcal^{k-1} \sigma_\Mcal,\\
			\sigma_0 \nu_F &= \begin{cases}
				\sigma_S &\mbox{if } F \in \Fcal, F \subset S \text{ for some } S \in \Sigma,\\
				\sigma_K \nu_F &\mbox{if } F\in \Fcal, F \subset K \text{ for some } K \in \mathcal{M} \text{ but } F \not\subset \partial K.
			\end{cases}
		\end{split}
	\end{align}
	Then $\Pi_\Mcal^k \div \sigma_0 = \Pi_\Mcal^k \nabla_u \psi_h(x,\Pi_\Mcal^k u_\Mcal)$.
\end{lemma}
The weights in \eqref{def:sigma-RT} are the degrees of freedom for Raviart-Thomas finite element functions \cite[Section 14.3]{ErnGuermond2021} and therefore, $\sigma_0$ is uniquely defined by \eqref{def:sigma-RT}.
\begin{proof}
	Given any $v_h = (v_\Mcal, v_\Sigma) \in V_\mathrm{D}(\Mcal)$ with $v_\Mcal \in P_k(\Mcal;\R^m)$, a piecewise integration by parts followed by the definition \eqref{def:divergence-reconstruction} of $\div_h$ leads to
	\begin{align*}
		- \int_\Omega \div_\pw \sigma_\Mcal \cdot v_\Mcal \d{x} = -\int_\Omega \div_h \sigma_h \cdot v_\Mcal \d{x} - \sum_{S \in \Sigma\setminus\Sigma_\mathrm{D}} \{v_\Mcal\}_S \cdot [\sigma_\Mcal \nu_S]_S \d{s}&\\
		+ \sum_{S \in \Sigma\setminus\Sigma_\mathrm{N}} [v_\Mcal]_S \cdot (\sigma_S - \{\sigma_\Mcal \nu_S\}_S) \d{s}&.
	\end{align*}
	This and the definition \eqref{def:gradient-reconstruction} of $\D_h$ imply
	\begin{align*}
		\int_\Omega \sigma_\Mcal : \D_h v_h \d{x} = - \int_\Omega \div_h \sigma_h \cdot v_\Mcal \d{x} + \sum_{S \in \Sigma\setminus\Sigma_\mathrm{D}} (v_S - \{v_\Mcal\}_S) \cdot [\sigma_\Mcal \nu_S]_S \d{s}&\\
		+ \sum_{S \in \Sigma\setminus\Sigma_\mathrm{N}} [v_\Mcal]_S \cdot (\sigma_S - \{\sigma_\Mcal \nu_S\}_S) \d{s}&.
	\end{align*}
	The combination of this with the discrete Euler-Lagrange equations \eqref{eq:dELE} yields
	\begin{align}
		\int_\Omega v_\Mcal \cdot (\nabla_u \psi_h(x, \Pi_\Mcal^k u_\Mcal) - \div_h \sigma_h) \d{x}
		+ \sum_{S \in \Sigma\setminus\Sigma_\mathrm{D}} \int_S (v_S - \{v_\Mcal\}_S) \cdot [\sigma_\Mcal \nu_S]_S \d{s}&\nonumber\\
		+ \sum_{S \in \Sigma\setminus\Sigma_\mathrm{N}} \int_S [v_\Mcal]_S \cdot (\sigma_S - \{\sigma_\Mcal \nu_S\}_S) \d{s} + \s(u_h; v_h) = 0&.
		\label{eq:dELE-v3}
	\end{align}
	Let $v_\Sigma$ be chosen such that $v_S \coloneqq \Pi_S^k\{v_\Mcal\}_S = \{v_\Mcal\}_S$ along $S \in \Sigma\setminus\Sigma_\mathrm{D}$ in \eqref{eq:dELE-v3}. For any $K \in \Mcal$ and $S \in \Sigma(K)$, $v_S - v_K$ equals $- [v_\Mcal]_S$ if $S \in \Sigma_\mathrm{D}$, $0$ if $S \in \Sigma_\mathrm{N}$, and $- \frac{\nu_S \cdot \nu_K|_S}{2} [v_\Mcal]_S$ if $S \in \Sigma(\Omega)$.
	Therefore,
	\begin{align*}
		\s(u_h;v_h) = - \sum_{S \in \Sigma(\Omega)} \frac{h_S^{-s}}{2} \int_S (T_{K_+,S} u_h - T_{K_-,S} u_h) \cdot \Pi_S^k[v_\Mcal]_S \d{s}&\\
		- \sum_{S \in \Sigma_\mathrm{D}} h_S^{-s} \int_S T_{K_+,S} u_h \cdot \Pi_S^k v_\Mcal \d{s}&.
	\end{align*}
	This, the definition of $\sigma_\Sigma$ in \Cref{lem:discrete-dual-variable}, and \eqref{eq:dELE-v3} conclude	
	\begin{align*}
		\int_\Omega v_\Mcal \cdot (\nabla_u \psi_h(x,\Pi_\Mcal^k u_\Mcal) - \div_h \sigma_h) \d{x} = 0 \quad\text{for any } v_\Mcal \in P_k(\Mcal;\R^m),
	\end{align*}
	whence $\div_h \sigma_h = \Pi_\Mcal^k \nabla_u \psi_h(x,\Pi_\Mcal^k u_\Mcal)$.
	Given any $v_k \in P_k(\mathcal{M};\mathbb{R}^m)$, a piecewise integration by parts and $\D_\pw v_k \in P_{k-1}(\Mcal;\M)$ show
	\begin{align*}
		\int_\Omega \div \sigma_0 \cdot v_k \d{x} = - \int_\Omega \sigma_\Mcal : \D_\pw v_k + \sum_{F \in \Fcal, F \subset \partial \Mcal} \int_F [v_k]_F \cdot \sigma_0 \nu_F \d{s}
	\end{align*}
	because the function $v_k$ only jumps on the skeleton $\partial \Mcal = \cup_{K \in \Mcal} \partial K$ of $\mathcal{M}$.
	This, the definition of $\sigma_0$ from \eqref{def:sigma-RT}, and of $\div_h$ from \eqref{def:divergence-reconstruction} imply
	\begin{align*}
		\int_\Omega \div \sigma_0 \cdot v_k \d{x} = - \int_\Omega \sigma_\Mcal : \D_\pw v_k + \sum_{S \in \Sigma\setminus\Sigma_\mathrm{N}} \int_S [v_k]_S \cdot \sigma_S \d{s} = \int_\Omega \div_h \sigma_h \cdot v_k \d{x}.
	\end{align*}
	The assertion then follows from $\div_h \sigma_h = \Pi_\Mcal^k \nabla_u \psi_h(x, \Pi_\Mcal^k u_\Mcal)$.
\end{proof}
In order to derive a~posteriori error estimate, we consider throughout the remaining parts of this paper the explicit representation \eqref{def:lower-order-terms} of the lower-order terms.
In particular, $\varphi \equiv 0$ in \eqref{ineq:assumption-cc-psi}.
The following arguments are well known in the literature \cite{Repin2000,Bartels2015,CarstensenLiu2015,CarstensenTran2021,BartelsKaltenbach2023} and sketched below for later reference.
Let $\widetilde{u}$ minimize
\begin{align*}
	E_{\sigma_0}(v) \coloneqq \int_\Omega (\Psi(\D v) + \div \sigma_0 \cdot v) \d{x} \quad\text{among } v \in V_\mathrm{D}.
\end{align*}
The corresponding dual energy reads
\begin{align*}
	E_{\sigma_0}^*(\tau) \coloneqq - \int_\Omega \Psi^*(\tau) \d{x} - \chi_{\div \sigma_0}(\div \tau) \quad\text{for any } \tau \in W_\mathrm{N}.
\end{align*}
Abbreviate $\widetilde{f} \coloneqq f + \div \sigma_0$.
The $L^2$ orthogonality $\widetilde{f} \perp P_k(\Mcal;\R^m)$ from Lemma \ref{lem:discrete-dual-variable}, a H\"older inequality, and a piecewise application of the Poincar\'e inequality imply
\begin{align*}
	E(u) - E_{\sigma_0}(u) = - \int_\Omega \widetilde{f} \cdot u \d{x} \leq c_\mathrm{P}\cnstS{cnst:Du}\osc(\widetilde{f},\Mcal)
\end{align*}
Given a conforming postprocessing $v_0 \in V_\mathrm{D}$, this and $E^*_{\sigma_0}(\sigma_0) \leq E_{\sigma_0}(\widetilde{u})$ prove
\begin{align}\label{ineq:LEB}
	\mathrm{LEB} \coloneqq E^*_{\sigma_0}(\sigma_0) - c_\mathrm{P}\cnstS{cnst:Du}\osc(\widetilde{f},\Mcal) \leq E(u) \leq E(v_0).
\end{align}
In practice, $v_0 \in S^{k+1}_\mathrm{D}(\Tcal;\R^m) \coloneqq P_{k+1}(\Tcal;\R^m) \cap V_\mathrm{D}$ is the nodal average of $u_\Mcal$ in a matching simplicial submesh $\Tcal$ of $\Mcal$.
If $\Psi$ satisfies \eqref{ineq:assumption-cc-W}, then the error arising from the coercivity of $\Psi$ can be bounded as follows. The arguments, that lead to \eqref{ineq:assmption-cc}, apply to the continuous level and prove $e(u,v_0) \leq E(v_0) - E(u)$ with $e(u,v_0) \coloneqq \Phi(\D u, \D v_0)$. This and
\eqref{ineq:LEB} imply the a~posteriori error control
\begin{align}\label{ineq:a-posteriori}
	e(u,v_0) \leq E(v_0) - E(u) \leq E(v_0) - \mathrm{LEB}.
\end{align}
Here, we follow \cite{BartelsKaltenbach2023} and consider the conforming postprocessing $v_0$ as the approximation for $u$.
Therefore, the quantity of interest is $e(u,v_0)$ instead of $\Phi(\D u, \D_h u_h)$, which avoids additional consistency terms.
In fact, the duality gap \eqref{ineq:a-posteriori} is, among other terms, part of the a~posteriori error estimators in \cite{CLiu2015,CarstensenTran2021}.
Thus, this approach leads to a more accurate error control.
In practice, a~posteriori error estimates for $\Phi(\D u, \D_h u_h)$ can be obtained from \eqref{ineq:a-posteriori} and a triangle inequality.
If $\Mcal$ is a regular triangulation into simplices, then $\osc(\widetilde{f}, \Mcal) = \osc(f,\Mcal)$.
Notice that the data oscillation $\osc(\widetilde{f},\Mcal)$ in \eqref{ineq:a-posteriori} may dominate the error for higher-order methods.
Since pointwise control over the divergence of the postprocessing $\sigma_0$ from \Cref{lem:discrete-dual-variable} is lost on polyhedral meshes, it will not vanish even if $f \in P_k(\Mcal;\R^m)$.
Therefore, the usefulness of this reconstruction is restricted in the latter case unless $\osc(\widetilde{f},\Mcal)$ scales quadratically. This is guaranteed for uniformly convex $W$ or in the case of the $p$-Laplace problem.

\begin{remark}[data oscillation]\label{rem:data-oscillation}
	If $\Psi$ satisfies \eqref{ineq:cc-primal}, then \eqref{ineq:a-posteriori} proves
	\begin{align*}
		e(u,v_0) \leq E(v_0) - E(u) \leq E_{\sigma_0}(v_0) - E_{\sigma_0}(u) + \int_\Omega \widetilde{f} \cdot (u - v_0)\d{x}
	\end{align*}
	This, a H\"older, a Poincar\'e, and a Young inequality imply
	\begin{align*}
		e(u,v_0)/p' \leq E_{\sigma_0}(v_0) - E_{\sigma_0}(u) + \cnstS{cnst:osc}\osc(\widetilde{f},\Mcal)^{p'}
	\end{align*}
	with the constant $\newcnstS\label{cnst:osc} \coloneqq c_\mathrm{P}^{p'}\cnstS{cnst:cc-primal}^{-1/(p-1)}/p'$.
	For the $p$-Laplace equation from \Cref{rem:p-Laplace}, the energy density $W$ satisfies \eqref{ineq:a-posteriori} with $\Phi(\alpha,\beta) \coloneqq \newcnstS\label{cnst:plaplace-qn}\|(|\alpha| + |\beta|)^{(p-2)/2} \times (\alpha - \beta)\|_2^2$ for some positive constant $\cnstS{cnst:plaplace-qn} > 0$ \cite{Tyukhtin1982,Chow1989,BarrettLiu1993}. This is known under the label \emph{quasi-norm} with different equivalent versions in \cite{CKlose2003,DieningKreuzer2008}.
	The properties of the N-shift functions \cite{DieningKreuzer2008} lead to, for any $\delta > 0$,
	\begin{align*}
		\int_\Omega \widetilde{f} \cdot (u - v_0) \d{x} \lesssim c_\delta \osc_p(\widetilde{f}, v_0, \Mcal)^2 + \delta \Phi(\D u, \D v).
	\end{align*}
	with $\osc_p(\widetilde{f}, v_0, \Mcal)^2 \coloneqq \sum_{K \in \Mcal} h_K^2 \int_K (|\nabla v_0|^{p-1} + h_K|\widetilde{f}|)^{p'-2}|\widetilde{f}|^2 \d{x}$ and a constant $c_\delta$ depending on $\delta$, cf.~\cite[Ineq.~(3.29)]{kaltenbach2022}.
	By choosing $\delta$ sufficiently small, we obtain the a~posteriori error estimate
	\begin{align*}
		e(u,v_0) \lesssim E_{\sigma_0}(v_0) - E_{\sigma_0}^*(\sigma_0) + \osc_p(\widetilde{f}, v_0, \Mcal)^2.
	\end{align*}
\end{remark}

\section{Numerical examples}\label{sec:numerical-examples}
This section tests the performance of the a~posteriori error control from \Cref{sec:a-posteriori} in three numerical benchmarks on the two dimensional L-shaped domain $\Omega \coloneqq (-1,1)^2 \setminus ([0,1) \times (-1,0])$ with pure Dirichlet boundary $\Gamma_\mathrm{D} = \partial \Omega$.
The initial triangulation in all benchmarks is displayed in \Cref{fig:volume_fraction}(a).
The computer experiments are carried out on regular triangulations into \emph{simplices}.

\subsection{Adaptive mesh-refining algorithm}
Up to data oscillation, the a~posteriori error estimator \eqref{ineq:a-posteriori} consists of the duality gap 
$E_{\sigma_0}(v_0) - E_{\sigma_0}^*(\sigma_0)$.
The localization of this has been discussed in \cite{BartelsKaltenbach2023} and we only restate some relevant details to the benchmarks below.
An integration by parts with $\div \sigma_0 = - f_h$ implies
\begin{align*}
	E_{\sigma_0}(v_0) - E^*_{\sigma_0}(\sigma_0) = \int_\Omega (\Psi(\nabla v_0) - \sigma_0 \cdot \nabla v_0 + \Psi^*(\sigma_0)) \d{x} \eqqcolon \sum_{K \in \Mcal} \eta(K)
\end{align*}
with the local refinement indicator
\begin{align}\label{def:refinement-indicator}
	\eta(K) \coloneqq \int_K (\Psi(\nabla v_0) - \sigma_0 \cdot \nabla v_0 + \Psi^*(\sigma_0)) \d{x}.
\end{align}
Notice from the Fenchel-Young inequality that $\Psi(\nabla v_0) - \sigma_0 \cdot \nabla v_0 + \Psi^*(\sigma_0) \geq 0$ holds pointwise a.e.~in $\Omega$ and so, $\eta(K) \geq 0$. 
The data oscillation is ignored unless it scales quadratically. In this case, its local contribution will be added to \eqref{def:refinement-indicator}.
Adaptive computations utilize the refinement indicator \eqref{def:refinement-indicator} in the standard adaptive mesh-refining loop \cite{CarstensenFeischlPagePraetorius2014} with the D\"orfler marking strategy, i.e., at each refinement step, a subset $\mathfrak{M} \subset \mathcal{M}$ with minimal cardinality is selected such that
\begin{align*}
	\sum\nolimits_{K \in \mathcal{M}} \eta(K) \leq \frac{1}{2}\sum\nolimits_{K \in \mathfrak{M}} \eta(K).
\end{align*}
The convergence history plots display the quantities of interest against the number of degrees of freedom $\mathrm{ndof}$ in a log-log plot.
(Recall the scaling $\mathrm{ndof} \approx h^{-2}_\mathrm{max}$ for uniform meshes.)
Solid lines indicate adaptive, while dashed lines are associated with uniform mesh refinements.

The discrete minimization problem \eqref{def:discrete-energy} is solved by an iterative solver \texttt{fminunc} from the MATLAB standard library in an extension of the data structures and the short MATLAB programs \cite{AlbertyCFunken1999}.
The first and (piecewise) second derivatives of $W$ have been provided for the trust-region quasi-Newton scheme with parameters of \texttt{fminunc} set to \texttt{FunctionTolerance} = \texttt{OptimalityTolerance} = \texttt{StepTolerance} =
$10^{-15}$ and \texttt{MaxIterations} = $10^3$ for improved accuracy.
The numerical integration of piecewise polynomials is carried out exactly.
For non-polynomial functions such as $W(\D_h v_h)$ with $v_h \in V_\mathrm{D}(\Mcal)$, the number of chosen quadrature points allows for exact integration of polynomials of degree at most $2pk + 1$ with the growth $p$ of $W$ and the polynomial order $k$ of the discretization.
\begin{figure}[ht]
	\begin{minipage}[b]{0.475\textwidth}
		\centering
		\includegraphics[height=4.6cm]{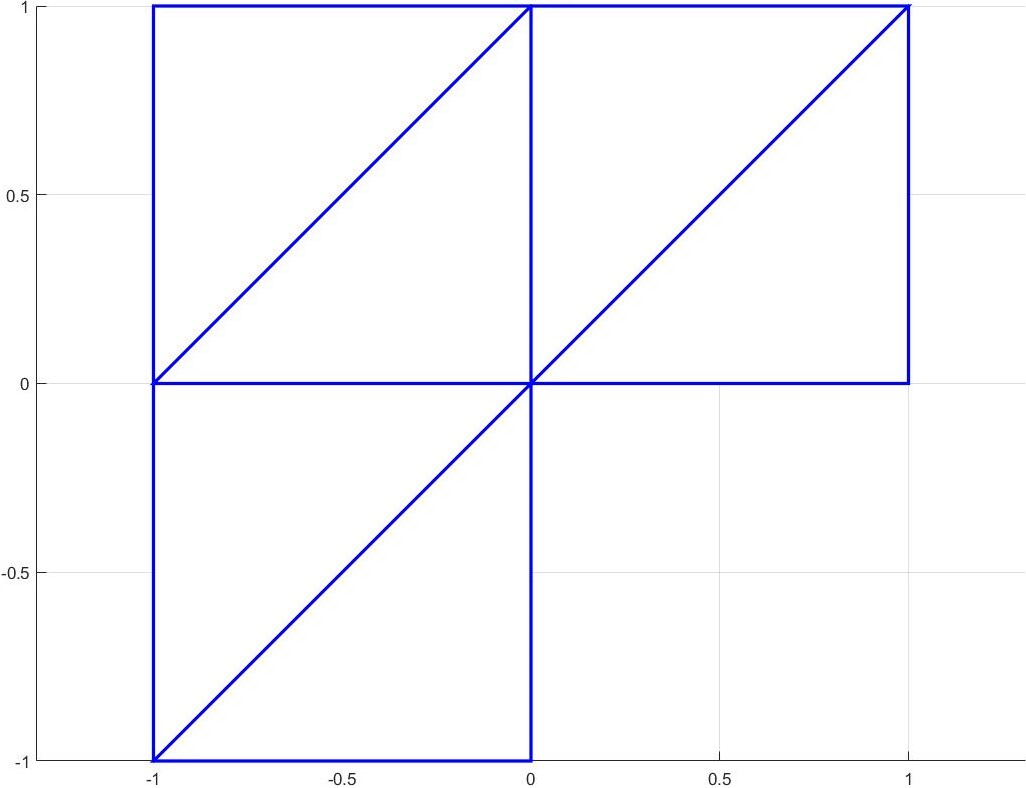}
	\end{minipage}\hfill
	\begin{minipage}[b]{0.475\textwidth}
		\centering
		\includegraphics[height=4.6cm]{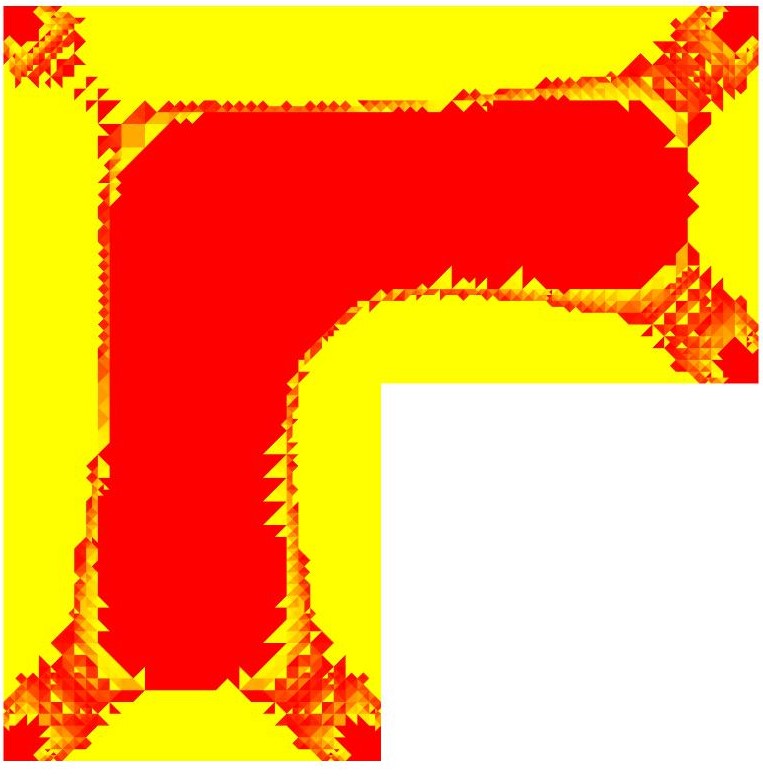}
	\end{minipage}
	\captionsetup{width=1\linewidth}
	\caption{(a) Initial triangulation of the L-shaped domain into 6 triangles and (b) material distribution in the optimal design problem of \Cref{sec:odp}}
	\label{fig:volume_fraction}
\end{figure}
\begin{figure}[ht]
	\begin{minipage}[b]{0.475\textwidth}
		\centering
		\includegraphics[height=5cm]{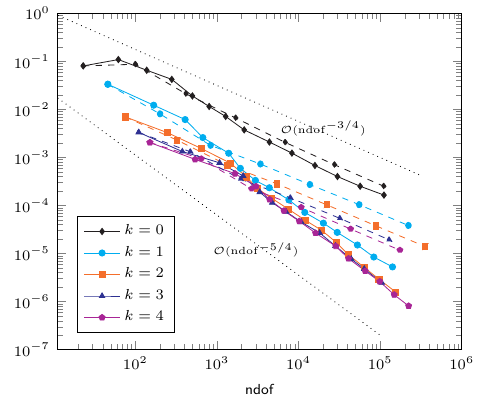}
	\end{minipage}\hfill
	\begin{minipage}[b]{0.475\textwidth}
		\centering
		\includegraphics[height=4.6cm]{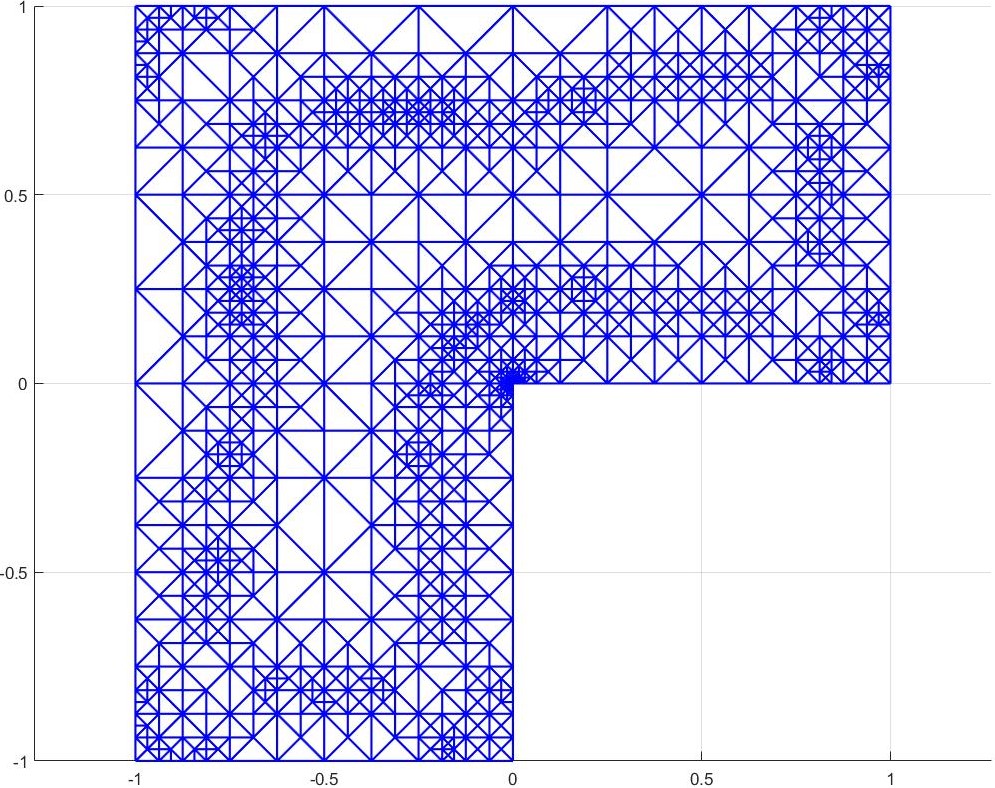}
	\end{minipage}
	\captionsetup{width=1\linewidth}
	\caption{(a) Convergence history plot of $E_{\sigma_0}(v_0) - E^*_{\sigma_0}(\sigma_0)$ for various $k$ and (b) adaptive triangulation into 2013 triangles obtained with $k = 2$ in \Cref{sec:odp}}
	\label{fig:conv_odp}
\end{figure}
\subsection{Optimal design problem}\label{sec:odp}
The optimal design problem seeks the optimal distribution of two materials with fixed amounts to fill a given domain for maximal torsion stiffness \cite{KohnStrang1986,BartelsC2008}. Given parameters $0 < t_1 < t_2$ and $0 < \mu_1 < \mu_2$ with $t_1 \mu_2 = \mu_1 t_2$, the energy density $\Psi(a) \coloneqq w(|a|)$, $a \in \mathbb{R}^n$, with
\begin{align*}
	w(t) \coloneqq \begin{cases}
		\mu_2 t^2/2 &\mbox{if } 0 \leq t \leq t_1,\\
		t_1\mu_2(t - t_1/2) &\mbox{if } t_1 \leq t \leq t_2,\\
		\mu_1 t^2/2 + t_1\mu_2(t_2/2 - t_1/2) &\mbox{if } t_2 \leq t
	\end{cases}
\end{align*}
satisfies \eqref{ineq:2-sided-growth} with $p = 2$ and \eqref{ineq:cc-dual} with $\cnstS{cnst:cc} = 1/(2\mu_2)$ \cite{BartelsC2008}.
This benchmark considers the parameters $\mu_1 = 1$, $\mu_2 = 2$, $t_1 = \sqrt{2\lambda\mu_1/\mu_2}$ for $\lambda = 0.0084$, $t_2 = \mu_2 t_1/\mu_1$ from \cite{BartelsC2008} and the input $r = 2$, $s = 1$ for the stabilization $\s_h$.
Since the data oscillation vanishes,
the a~posteriori error estimate from \eqref{ineq:a-posteriori} reads
\begin{align*}
	e(u,v_0) \coloneqq \frac{1}{2\mu}\|\nabla \Psi(\nabla u) - \nabla \Psi(\nabla v_0)\|_2^2 \leq E_{\sigma_0}(v_0) - E_{\sigma_0}^*(\sigma_0) = E(v_0) - E^*(\sigma_0).
\end{align*}
The approximated material distribution obtained by adaptive computation with $k = 0$ is displayed in \Cref{fig:volume_fraction}(b) using the volume fraction plot from \cite[Section 5]{BartelsC2008}. 
We observe two homogenous phases and a transition layer with a fine mixture of the two materials.
On uniform meshes, the convergence rate $2/3$ for $E_{\sigma_0}(v_0) - E_{\sigma_0}^*(\sigma_0)$ is observed in \Cref{fig:conv_odp}(a).
The adaptive mesh-refining algorithm refines towards the singularity at the origin and the transition layer as displayed \Cref{fig:conv_odp}(b).
This leads to slightly improved convergence rates for $E_{\sigma_0}(v_0) - E_{\sigma_0}^*(\sigma_0)$.

\begin{figure}[ht]
	\begin{minipage}[t]{0.475\textwidth}
		\centering
		\includegraphics[height=5cm]{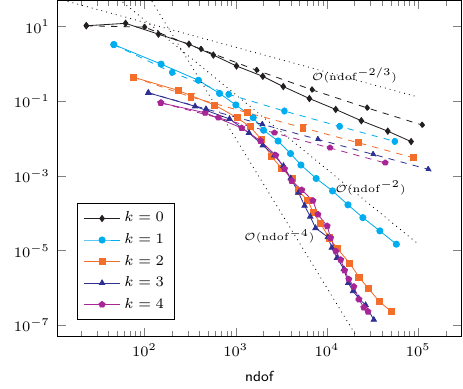}
	\end{minipage}\hfill
	\begin{minipage}[t]{0.475\textwidth}
		\centering
		\includegraphics[height=5cm]{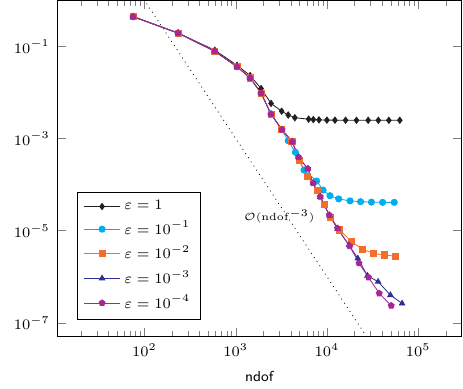}
	\end{minipage}
	\captionsetup{width=1\linewidth}
	\caption{Convergence history plot of $E_{\sigma_0^\varepsilon}(v_0^\varepsilon) - E_{\sigma_0^\varepsilon}^*(\sigma_0^\varepsilon)$ for (a) various $k$ and $\varepsilon = 10^{-4}$ and (b) $k = 2$ and various $\varepsilon$ in \Cref{sec:Bingham-flow}}
	\label{fig:conv_bingham}
\end{figure}

\subsection{Bingham flow in a pipe}\label{sec:Bingham-flow}
Given fixed positive parameters $\mu, g > 0$, the modelling of a uni-directional flow through a pipe with cross-section $\Omega \subset \mathbb{R}^2$ leads to the minimization problem \eqref{def:energy} with the energy density
\begin{align*}
	\Psi(a) \coloneqq \mu|a|^2/2 + g|a| \quad\text{for any } a \in \mathbb{R}^2,
\end{align*} 
cf.~\cite{CarstenReddySchedensack2016} for further details.
Given $\alpha \in \mathbb{R}^2 \setminus \{0\}$, the radial symmetry of $\Psi$ leads to
\begin{align*}
	\Psi^*(\alpha) = \sup_{a \in \mathbb{R}^2} (\alpha \cdot a - \Psi(a)) = \sup_{t \geq 0} (t|\alpha|^2 - \mu t^2 |\alpha|^2/2 - g t |\alpha|) \eqqcolon \varphi_\alpha(t).
\end{align*}
If $|\alpha| \leq g$, then $\varphi_\alpha(t) \leq 0$ and so, $\Psi^*(\alpha) = \sup_{t \geq 0} \varphi_\alpha(t) = \varphi_\alpha(0) = 0$.
If $|\alpha| > g$, then the maximum of $\varphi_\alpha(t)$ in $[0,\infty)$ is attained at $t = (|\alpha| - g)/(\mu|\alpha|)$. Hence,
\begin{align*}
	\Psi^*(\alpha) = \begin{cases}
		0 &\mbox{if } |\alpha| \leq g\\
		(|\alpha| - g)^2/(2 \mu) &\mbox{if } |\alpha| > g.
	\end{cases}
\end{align*}
The strict convexity of $\Psi$ leads to a unique the minimizer $u$ of $E$ in $V_\mathrm{D}$. Although $\Psi$ is not differentiable, there exists $\sigma \in H(\div,\Omega) = W^2(\div,\Omega)$ such that $\sigma \in \partial \Psi(\nabla u)$ and $\div \sigma = -f$ pointwise a.e.~in $\Omega$ \cite[Chapter II, Theorem 6.3]{Glowinski2008}.
Thus, there is no duality gap $E(u) = E^*(\sigma)$.
Furthermore, the minimizer $u$ and any $v \in V_\mathrm{D}$ satisfy $e(u,v) \leq E(v) - E(u)$ for any $v \in V_\mathrm{D}$ with $e(u,v) \coloneqq \frac{\mu}{2}\|\nabla(u - v)\|_2^2$ \cite[Lemma 1]{CarstenReddySchedensack2016}.
Thus, \Cref{rem:data-oscillation} applies and leads, for any $\tau \in W_\mathrm{N}$, to
\begin{align}\label{ineq:a-post-bingham}
	e(u,v)/2 \leq E_{\tau}(v) - E^*_\tau(\tau) + \cnstS{cnst:osc}\osc(\widetilde{f}, \Mcal)^2.
\end{align}
The postprocessings for \eqref{ineq:a-post-bingham} are obtained from a regularized discrete problem as in \cite{CarstenReddySchedensack2016}. Given $\varepsilon > 0$, define $\Psi_\varepsilon \in C^1(\mathbb{R}^2)$ by
\begin{align*}
	\Psi_\varepsilon(a) \coloneqq \mu|a|^2/2 + g(\sqrt{|a|^2 + \varepsilon^2}) \quad\text{for any } a \in \mathbb{R}^2.
\end{align*}
The unique minimizer $u_h^\varepsilon = (u_\Mcal^\varepsilon, u_\Sigma^\varepsilon)$ of the discrete energy
\begin{align*}
	E_h^\varepsilon(v_h) \coloneqq \int_\Omega (\Psi_\varepsilon(\D_h v_h) - f_h v_\Mcal) \d{x} + \s_h(v_h)/2 \quad\text{among } v_h = (v_\Mcal, v_\Sigma) \in V_\mathrm{D}(\Mcal)
\end{align*}
allows for the postprocessings $\sigma_0^\varepsilon \in W_\mathrm{N}$ with $\div \sigma_0^\varepsilon = -f_h$ from \Cref{lem:discrete-dual-variable} and $v_0^\varepsilon \in S^{k+1}_\mathrm{D}(\Mcal)$ as the nodal average of $u_\Mcal^\varepsilon$. From \eqref{ineq:a-post-bingham}, we deduce
\begin{align*}
	e(u,v_0^\varepsilon)/2 \leq E_{\sigma_0^\varepsilon}(v_0^\varepsilon) - E_{\sigma_0^\varepsilon}^*(\sigma_0^\varepsilon) + \cnstS{cnst:osc}\osc(f,\Mcal)^2.
\end{align*}
In the following computer experiment, we set $\mu = 1$, $g = 0.2$, $f \equiv 10$ with the reference minimal energy $E(u) = -9.32049$ obtained from the computable bounds \eqref{ineq:LEB} in adaptive computations and $r = 2$, $s = 1$ for the stabilization $\s_h$.
On uniform meshes, the convergence rate $2/3$ for $E(v_0^\varepsilon) - E^*(\sigma_0^\varepsilon)$ is observed in \Cref{fig:conv_bingham}(a) for all polynomial degrees and $\varepsilon = 10^{-4}$.
The adaptive algorithm refines towards the expected singularity at the origin and leads to the optimal convergence rates $k+1$ for all displayed polynomial degree $k$ until stagnation due to the regularization.
The effect of the latter is displayed in \Cref{fig:conv_bingham}(b) for the polynomial degree $k = 2$ and parameters $\varepsilon = 10^0, 10^{-1}, \dots, 10^{-4}$.

\begin{figure}[ht]
	\begin{minipage}[t]{0.475\textwidth}
		\centering
		\includegraphics[height=5cm]{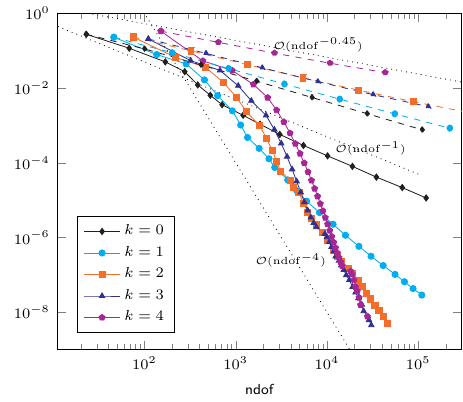}
	\end{minipage}\hfill
	\begin{minipage}[t]{0.475\textwidth}
		\centering
		\includegraphics[height=5cm]{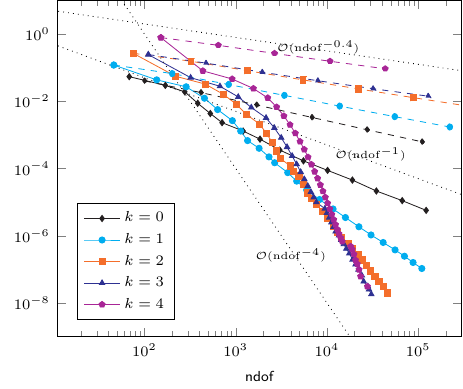}
	\end{minipage}
	\captionsetup{width=1\linewidth}
	\caption{Convergence history plot of (a) RHS and (b) $\|\sqrt{\varrho}\nabla(u - v_0)\|^2_2$ for various $k$ in \Cref{sec:plaplace}}
	\label{fig:conv_plaplace_1}
\end{figure}
\begin{figure}[ht]
	\begin{minipage}[t]{0.475\textwidth}
		\centering
		\includegraphics[height=5cm]{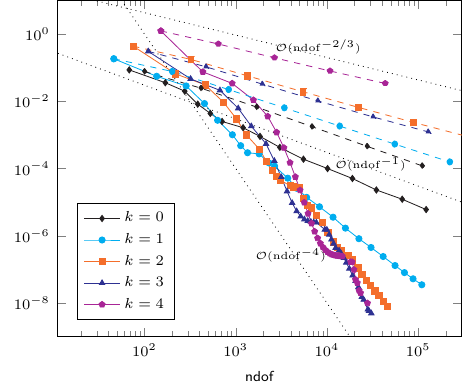}
	\end{minipage}\hfill
	\begin{minipage}[t]{0.475\textwidth}
		\centering
		\includegraphics[height=5cm]{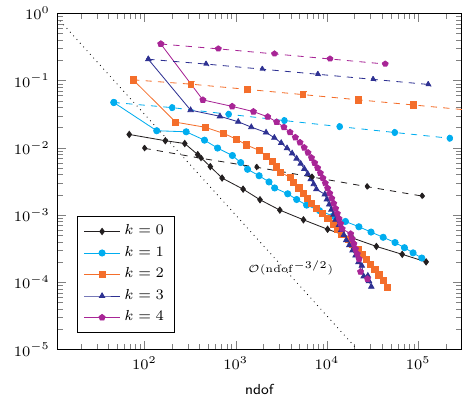}
	\end{minipage}
	\captionsetup{width=1\linewidth}
	\caption{Convergence history plot of (a) $\|\sigma - \nabla \Psi(\nabla v_0)\|_{4/3}^2$ and (b) $\|\nabla(u - v_0)\|_4^2$ for various $k$ in \Cref{sec:plaplace}}
	\label{fig:conv_plaplace_2}
\end{figure}

\subsection{$p$-Laplace problem}\label{sec:plaplace}
In this final benchmark, we consider the $4$-Laplace problem from \Cref{rem:p-Laplace} with the exact solution
\begin{align*}
	u(r,\varphi) = r^{7/8}\sin(7\varphi/8) \in W^{1,4}(\Omega)
\end{align*}
and the right-hand side
\begin{align*}
	f(r,\varphi) = (7/8)^{3/4}r^{-11/8}\sin(7\varphi/8) \in L^{16/11-\varepsilon}(\Omega)
\end{align*}
for any $\varepsilon > 0$ from \cite{CKlose2003}.
Some remarks are in order due to the inhomogenous Dirichlet data prescribed by $u$.
The discrete problem minimizes \eqref{def:discrete-energy} in the affine space $\I_V u + V_\mathrm{D}(\Mcal)$.
Since the construction of $\sigma_0$ in \Cref{lem:discrete-dual-variable} only relies on the discrete Euler-Lagrange equations, it applies verbatim.
To obtain a postprocessing $v_0 \in u + V_\mathrm{D}$, let $\widetilde{v}_0 \in S^{k+1}_\mathrm{D}(\Mcal)$ denote the nodal average of $u_\Mcal - \Pi_\Mcal^{k+1} u$. We set $v_0 \coloneqq \widetilde{v}_0 + u$.
(Notice that, in the computation of $v_0$, $u$ can be replaced by any extension of the given Dirichlet data $u|_{\partial \Omega}$.)
The coercivity of $\Psi$ from \Cref{rem:p-Laplace} and \Cref{rem:data-oscillation} proves
\begin{align*}
	&e(u,v_0) \coloneqq \|\nabla(u - v_0)\|_4^4 + \|\sigma - \nabla \Psi(\nabla v_0)\|_{4/3}^2\\
	&\qquad + \|\sqrt{\varrho}\nabla(u - v_0)\|^2_2 \lesssim E_{\sigma_0}(v_0) - E^*_{\sigma_0}(\sigma_0) + \osc_4(f,v_0,\Mcal)^2 \eqqcolon \mathrm{RHS}
\end{align*}
with $\varrho \coloneqq (|\nabla u| + |\nabla v_0|)^{2}$ and the dual energy
\begin{align*}
	E_{\sigma_0}^*(\sigma_0) = - \int_\Omega \Psi^*(\sigma_0) \d{x} + \int_{\Gamma_\mathrm{D}} u \sigma_0 \cdot \nu \d{s}.
\end{align*}
In view of the analysis in \cite{DieningKreuzer2008,kaltenbach2022} for lowest-order methods, we expect $h_{\max}^{2(k+1)}$ as the best possible convergence rate for the primal-dual gap $E_{\sigma_0}(v_0) - E_{\sigma_0}^*(\sigma_0)$.
To achieve this, the stabilization $\s_h$ should -- at least in theory -- scale quadratically. This motivates the input $r = 2$ and $s = 1$ for $\s_h$ in this benchmark.
\Cref{fig:conv_plaplace_1}--\ref{fig:conv_plaplace_2} display a convergence rate between $0.45$ and $2/3$ for RHS, $\|\sqrt{\varrho}\nabla(u - v_0)\|_2^2$, and $\|\sigma - \nabla \Psi(\nabla v_0)\|_{4/3}^2$, but only a marginal improvement of $\|\nabla(u - v_0)\|_4^2$ in the preasymtotic range on uniform meshes.
Adaptive computation refines towards the singularity at the origin and recovers the optimal convergence rates $k+1$ for RHS, $\|\sqrt{\varrho}\nabla(u - v_0)\|_2^2$, and $\|\sigma - \nabla W(\nabla v_0)\|_{4/3}^2$ with all displayed $k$.
For the displacement error $\|\nabla(u - v_0)\|_4^2$, we observe the empirical convergence rates $(k+1)/2$ for $k \geq 2$.
In undisplayed computer experiments with the input $r = p = 4$, the optimal convergence rates $2(k+1)$ is observed for all displayed errors of this subsection for the choice $s \geq (p-2)(k+1) + p - 1$. (Notice from \eqref{ineq:proof-convergence-rates-2} and \eqref{ineq:convergence-rates-stab} that $s = (p-2)(k+1) + p - 1$ is the unique parameter with $\s_h(\I_V u) + \gamma_h(\I_W \sigma) \lesssim h_{\max}^{2(k+1)}$ for smooth $u$ and $\sigma$).

\subsection{Conclusions}
In all computer experiments, adaptive computation leads to improved convergence rates for the a~posteriori error estimator compared to uniform mesh refinements.
Higher polynomial degrees provide additional improvement, which depends on the model problem. For the Bingham flow and the $p$-Laplace problem, we observe the optimal convergence rates $\mathrm{ndof}^{-(k+1)}$. This appears to be sensitive to the choice of the parameters $r$ and $s$ in the stabilization $\s_h$ if $p \neq 2$.

\bibliographystyle{amsplain}
\bibliography{references.bib}

\providecommand{\bysame}{\leavevmode\hbox to3em{\hrulefill}\thinspace}
\providecommand{\MR}{\relax\ifhmode\unskip\space\fi MR }
\providecommand{\MRhref}[2]{%
  \href{http://www.ams.org/mathscinet-getitem?mr=#1}{#2}
}
\providecommand{\href}[2]{#2}
\begin{thebibliography}{10}

\bibitem{AbbasErnPignet2018}
M.~Abbas, A.~Ern, and N.~Pignet, \emph{Hybrid high-order methods for finite
  deformations of hyperelastic materials}, Comput. Mech. \textbf{62} (2018),
  no.~4, 909--928.

\bibitem{AlbertyCFunken1999}
J.~Alberty, C.~Carstensen, and S.~A. Funken, \emph{Remarks around 50 lines of
  {M}atlab: short finite element implementation}, Numer. Algorithms \textbf{20}
  (1999), no.~2-3, 117--137.

\bibitem{BarrettLiu1993}
J.~W. Barrett and W.~B. Liu, \emph{Finite element approximation of the
  {$p$}-{L}aplacian}, Math. Comp. \textbf{61} (1993), no.~204, 523--537.

\bibitem{Bartels2015}
S.~Bartels, \emph{Numerical methods for nonlinear partial differential
  equations}, Springer Series in Computational Mathematics, vol.~47, Springer,
  Cham, 2015.

\bibitem{Bartels2021}
\bysame, \emph{Error estimates for a class of discontinuous {G}alerkin methods
  for nonsmooth problems via convex duality relations}, Math. Comp. \textbf{90}
  (2021), no.~332, 2579--2602.

\bibitem{Bartels2021b}
\bysame, \emph{Nonconforming discretizations of convex minimization problems
  and precise relations to mixed methods}, Comput. Math. Appl. \textbf{93}
  (2021), 214--229.

\bibitem{BartelsC2008}
S.~Bartels and C.~Carstensen, \emph{A convergent adaptive finite element method
  for an optimal design problem}, Numer. Math. \textbf{108} (2008), no.~3,
  359--385.

\bibitem{BartelsKaltenbach2023}
S.~Bartels and A.~Kaltenbach, \emph{Explicit and efficient error estimation for
  convex minimization problems}, Math. Comp. \textbf{92} (2023), no.~343,
  2247--2279.

\bibitem{Bebendorf2003}
M.~Bebendorf, \emph{A note on the {P}oincar\'{e} inequality for convex
  domains}, Z. Anal. Anwendungen \textbf{22} (2003), no.~4, 751--756.

\bibitem{BottiDiPietroSochala2017}
Michele Botti, Daniele~A. Di~Pietro, and Pierre Sochala, \emph{A hybrid
  high-order method for nonlinear elasticity}, SIAM J. Numer. Anal. \textbf{55}
  (2017), no.~6, 2687--2717. \MR{3721565}

\bibitem{BraessSchoeberl2008}
D.~Braess and J.~Sch\"{o}berl, \emph{Equilibrated residual error estimator for
  edge elements}, Math. Comp. \textbf{77} (2008), no.~262, 651--672.

\bibitem{CarstensenFeischlPagePraetorius2014}
C.~Carstensen, M.~Feischl, M.~Page, and D.~Praetorius, \emph{Axioms of
  adaptivity}, Comput. Math. Appl. \textbf{67} (2014), no.~6, 1195--1253.

\bibitem{CKlose2003}
C.~Carstensen and R.~Klose, \emph{A~posteriori finite element error control for
  the p-{L}aplace problem}, SIAM J. Sci. Comput. \textbf{25} (2003), 792--814.

\bibitem{CarstensenLiu2015}
C.~Carstensen and D.~J. Liu, \emph{Nonconforming {FEM}s for an optimal design
  problem}, SIAM J. Numer. Anal. \textbf{53} (2015), no.~2, 874--894.

\bibitem{CLiu2015}
\bysame, \emph{Nonconforming {FEM}s for an optimal design problem}, SIAM J.
  Numer. Anal. \textbf{53} (2015), no.~2, 874--894.

\bibitem{CPlechac1997}
C.~Carstensen and P.~Plech\'a\v{c}, \emph{Numerical solution of the scalar
  double-well problem allowing microstructure}, Math. Comp. \textbf{66} (1997),
  no.~219, 997--1026.

\bibitem{CarstenReddySchedensack2016}
C.~Carstensen, B.~D. Reddy, and M.~Schedensack, \emph{A natural nonconforming
  {FEM} for the {B}ingham flow problem is quasi-optimal}, Numer. Math.
  \textbf{133} (2016), no.~1, 37--66.

\bibitem{CarstensenTran2021}
C.~Carstensen and N.~T. Tran, \emph{Unstabilized {H}ybrid {H}igh-order {M}ethod
  for a {C}lass of {D}egenerate {C}onvex {M}inimization {P}roblems}, SIAM J.
  Numer. Anal. \textbf{59} (2021), no.~3, 1348--1373.

\bibitem{CarstensenTran2022}
\bysame, \emph{Convergent adaptive hybrid higher-order schemes for convex
  minimization}, Numer. Math. \textbf{151} (2022), no.~2, 329--367.

\bibitem{Chow1989}
S.-S. Chow, \emph{Finite element error estimates for non-linear elliptic
  equations of monotone type}, Numer. Math. \textbf{54} (1989), 373--393.

\bibitem{ChuaWeeden2006}
S.-K. Chua and R.~L. Wheeden, \emph{Estimates of best constants for weighted
  {P}oincar\'{e} inequalities on convex domains}, Proc. London Math. Soc. (3)
  \textbf{93} (2006), no.~1, 197--226.

\bibitem{CockburnDiPietroErn2016}
B.~Cockburn, D.~A. Di~Pietro, and A.~Ern, \emph{Bridging the hybrid high-order
  and hybridizable discontinuous {G}alerkin methods}, ESAIM Math. Model. Numer.
  Anal. \textbf{50} (2016), no.~3, 635--650.

\bibitem{DiPietroDroniou2017}
D.~A. Di~Pietro and J.~Droniou, \emph{A hybrid high-order method for
  {L}eray-{L}ions elliptic equations on general meshes}, Math. Comp.
  \textbf{86} (2017), no.~307, 2159--2191.

\bibitem{DiPietroDroniou2020}
\bysame, \emph{The hybrid high-order method for polytopal meshes}, MS\&A.
  Modeling, Simulation and Applications, vol.~19, Springer, Cham, 2020, Design,
  analysis, and applications.

\bibitem{DiPietroDroniou2021}
D.~A. Di~Pietro, J.~Droniou, and A.~Harnist, \emph{Improved error estimates for
  hybrid high-order discretizations of {L}eray-{L}ions problems}, Calcolo
  \textbf{58} (2021), no.~2, Paper No. 19, 24.

\bibitem{DiPietroErn2012}
D.~A. Di~Pietro and A.~Ern, \emph{Mathematical aspects of discontinuous
  {G}alerkin methods}, Math\'{e}matiques \& Applications, vol.~69, Springer,
  Heidelberg, 2012.

\bibitem{DiPietroErn2015}
\bysame, \emph{A hybrid high-order locking-free method for linear elasticity on
  general meshes}, Comput. Methods Appl. Mech. Engrg. \textbf{283} (2015),
  1--21.

\bibitem{DiPietroErnLemaire2014}
D.~A. Di~Pietro, A.~Ern, and S.~Lemaire, \emph{An arbitrary-order and
  compact-stencil discretization of diffusion on general meshes based on local
  reconstruction operators}, Comput. Methods Appl. Math. \textbf{14} (2014),
  no.~4, 461--472.

\bibitem{DiPietroDroniou2023}
Daniele~A. Di~Pietro and J\'{e}r\^{o}me Droniou, \emph{An arbitrary-order
  discrete de {R}ham complex on polyhedral meshes: exactness, {P}oincar\'{e}
  inequalities, and consistency}, Found. Comput. Math. \textbf{23} (2023),
  no.~1, 85--164. \MR{4546145}

\bibitem{DiPietroDroniouRapetti2020}
Daniele~A. Di~Pietro, J\'{e}r\^{o}me Droniou, and Francesca Rapetti,
  \emph{Fully discrete polynomial de {R}ham sequences of arbitrary degree on
  polygons and polyhedra}, Math. Models Methods Appl. Sci. \textbf{30} (2020),
  no.~9, 1809--1855. \MR{4151796}

\bibitem{DieningKoenerRuzickaToulopoulos2014}
L.~Diening, D.~K\"{o}ner, M.~R{$\mathring{\text{u}}$}\v{z}i\v{c}ka, and
  I.~Toulopoulos, \emph{A local discontinuous {G}alerkin approximation for
  systems with {$p$}-structure}, IMA J. Numer. Anal. \textbf{34} (2014), no.~4,
  1447--1488.

\bibitem{DieningKreuzer2008}
L.~Diening and C.~Kreuzer, \emph{Linear convergence of an adaptive finite
  element method for the {$p$}-{L}aplacian equation}, SIAM J. Numer. Anal.
  \textbf{46} (2008), no.~2, 614--638.

\bibitem{DieningRuzicka2007}
L.~Diening and M.~R{$\mathring{\text{u}}$}\v{z}i\v{c}ka, \emph{Interpolation
  operators in {O}rlicz-{S}obolev spaces}, Numer. Math. \textbf{107} (2007),
  no.~1, 107--129.

\bibitem{ErnGuermond2021}
A.~Ern and J.-L. Guermond, \emph{Finite elements {I}---{A}pproximation and
  interpolation}, Texts in Applied Mathematics, vol.~72, Springer, Cham, 2021.

\bibitem{ErnVohralik2015}
A.~Ern and M.~Vohral\'{\i}k, \emph{Polynomial-degree-robust a posteriori
  estimates in a unified setting for conforming, nonconforming, discontinuous
  {G}alerkin, and mixed discretizations}, SIAM J. Numer. Anal. \textbf{53}
  (2015), no.~2, 1058--1081.

\bibitem{ErnVohralik2020}
\bysame, \emph{Stable broken {$H^1$} and {$H({\rm div})$} polynomial extensions
  for polynomial-degree-robust potential and flux reconstruction in three space
  dimensions}, Math. Comp. \textbf{89} (2020), no.~322, 551--594.

\bibitem{Glowinski2008}
R.~Glowinski, \emph{Numerical methods for nonlinear variational problems},
  Scientific Computation, Springer-Verlag, Berlin, 2008, Reprint of the 1984
  original.

\bibitem{GlowinskiMarrocco1975}
R.~Glowinski and A.~Marrocco, \emph{Sur l'approximation, par \'{e}l\'{e}ments
  finis d'ordre un, et la r\'{e}solution, par p\'{e}nalisation-dualit\'{e},
  d'une classe de probl\`emes de {D}irichlet non lin\'{e}aires}, Rev.
  Fran\c{c}aise Automat. Informat. Recherche Op\'{e}rationnelle S\'{e}r. Rouge
  Anal. Num\'{e}r. \textbf{9} (1975), no.~no. {\rm R}-2, 41--76.

\bibitem{kaltenbach2022}
A.~Kaltenbach, \emph{Error analysis for a crouzeix-raviart approximation of the
  $ p $-dirichlet problem}, arXiv:2210.12116 (2022).

\bibitem{KohnStrang1986}
R.~V. Kohn and G.~Strang, \emph{Optimal design and relaxation of variational
  problems. {I}}, Comm. Pure Appl. Math. \textbf{39} (1986), no.~1, 113--137.

\bibitem{Marini1985}
L.~D. Marini, \emph{An inexpensive method for the evaluation of the solution of
  the lowest order {R}aviart-{T}homas mixed method}, SIAM J. Numer. Anal.
  \textbf{22} (1985), no.~3, 493--496.

\bibitem{Repin2000}
S.~I. Repin, \emph{A posteriori error estimation for variational problems with
  uniformly convex functionals}, Math. Comp. \textbf{69} (2000), no.~230,
  481--500.

\bibitem{Rockafellar1970}
R.~T. Rockafellar, \emph{Convex analysis}, Princeton Mathematical Series, No.
  28, Princeton University Press, Princeton, N.J., 1970.

\bibitem{Tyukhtin1982}
V.~B. Tyukhtin, \emph{The rate of convergence of approximation methods for
  solving one-sided variational problems. {I}}, Teoret. Mat. Fiz. \textbf{51}
  (1982), no.~2, 111--113, 121.

\end{thebibliography}

\end{document}